\documentclass{article}
\usepackage{graphicx} 
\usepackage{amsthm}
\usepackage{amsfonts, amsmath, amssymb, amsthm}
\usepackage{hyperref}
\usepackage{cleveref}
\usepackage[utf8x]{inputenc}
\usepackage{url}
\usepackage{tikz-cd}

\newcommand{\Psf}{\mathsf{P}}
\newcommand{\Qsf}{\mathsf{Q}}

\def\F{\mathcal{F}}

\def\P{\mathcal{P}}

\newcommand{\uh}{\upharpoonright}

\newcommand{\qvdash}{\operatorname{{?}{\vdash}}}
\newcommand{\nqvdash}{\operatorname{{?}{\nvdash}}}

\newcommand{\RCA}[0]{\mathsf{RCA}}
\newcommand{\WKL}[0]{\mathsf{WKL}}
\newcommand{\WWKL}[0]{\mathsf{WWKL}}
\newcommand{\ACA}[0]{\mathsf{ACA}}
\newcommand{\ATR}[0]{\mathsf{ATR}}

\newcommand{\RRT}[0]{\mathsf{RRT}}
\newcommand{\FS}[0]{\mathsf{FS}}
\newcommand{\PFS}[0]{\mathsf{PFS}}
\newcommand{\TS}[0]{\mathsf{TS}}

\newcommand{\RT}[0]{\mathsf{RT}}

\newcommand{\NN}[0]{\mathbb{N}}

\newcommand{\W}{\mathrm{W}}
\newcommand{\sW}{\mathrm{sW}}

\newcommand{\lex}{\mathrm{lex}}

\newcommand{\base}{\operatorname{base}}

\def\qt#1{``#1''}%

\title{Ramsey-like theorems for the Schreier barrier}
\date{\today}

\newtheorem{theorem}{Theorem}
\numberwithin{theorem}{section}

\newtheorem{lemma}[theorem]{Lemma}
\newtheorem{question}[theorem]{Question}
\newtheorem{proposition}[theorem]{Proposition}
\newtheorem{remark}[theorem]{Remark}
\newtheorem{definition}[theorem]{Definition}
\newtheorem{corollary}[theorem]{Corollary}

\newcommand{\R}{\mathcal{R}}

\newcommand{\T}{\mathcal{T}}

\makeatletter
\newtheorem*{rep@theorem}{\rep@title}
\newcommand{\newreptheorem}[2]{%
\newenvironment{rep#1}[1]{%
 \def\rep@title{#2 \ref{##1}}%
 \begin{rep@theorem}}%
 {\end{rep@theorem}}}
\makeatother

\newreptheorem{theorem}{Theorem}

\usepackage{xcolor}
\newcommand{\ludovic}[1]{\textcolor{blue}{#1}}

\newcommand{\lorenzo}[1]{\textcolor{orange}{#1}}

\author{Lorenzo Carlucci \and Oriola Gjetaj \and Quentin Le Houérou \and Ludovic Levy Patey}

\begin{document}

\maketitle

\begin{abstract}
The family of finite subsets $s$ of the natural numbers such that $|s|=1+\min s$ is known as the Schreier barrier in combinatorics and Banach Space theory, and as the family of exactly $\omega$-large sets in Logic. We formulate and prove the generalizations of Friedman's Free Set and Thin Set theorems and of Rainbow Ramsey's theorem to colorings of the Schreier barrier. We analyze the strength of these theorems from the point of view of Computability Theory and Reverse Mathematics. Surprisingly, the exactly $\omega$-large counterparts of the Thin Set and Free Set theorems can code $\emptyset^{(\omega)}$, while the exactly $\omega$-large Rainbow Ramsey theorem does not code the halting set.
\end{abstract}

\section{Introduction and motivation}

Ramsey's theorems have been widely investigated from the point of view of Computability Theory, Proof Theory and Reverse Mathematics (see \cite{hirschfeldt2015slicing} for details and references).
In his seminal paper, Jockusch \cite{JockuschRamsey_1972} gave a deep analysis of Ramsey's theorem using tools from Computability Theory, which established this theorem as an important bridge between Combinatorics and Computability.
The effective and logical strength of many consequences and variants of Ramsey's theorem have since been investigated. Among those, the Free Set, Thin Set and Rainbow Ramsey theorems have attracted significant interest in recent decades (see, e.g., \cite{Cholak_Giusto_Hirst_Jockusch_2005, wang2014some, Pat:16, cholak2020thin, Liu2022-LIUTRM}), due to the peculiar behavior of these theorems when compared to Ramsey's theorem.
The Free Set Theorem (denoted $\FS^n$), introduced in the context of Reverse Mathematics by Harvey Friedman \cite{Fri-Sim:00}, states that for every coloring of the $n$-subsets of the natural numbers in unboundedly many colors, there exists an infinite set $H$ of natural numbers such that for all $n$-subsets $s$ of $H$, the color of $s$ is either not in $H$ or else is in $s$ itself. Such a set is called free for the coloring. The Thin Set Theorem (denoted $\TS^n_\omega$ or $\TS^n$) is a weak variant of the Free Set Theorem asserting that for any coloring of the $n$-subsets of the natural numbers there is an infinite set $H$ of natural numbers such that the $n$-subsets of $H$ avoid at least one color. The Rainbow Ramsey Theorem (denoted $\RRT^n_k$) asserts that for every coloring of the $n$-subsets of the natural numbers in which each color is used at most $k$ times, there is an infinite set $H$ of natural numbers such that the coloring assigns different colors to different $n$-subsets of $H$. 
While Ramsey's theorem for colorings of $3$-subsets already codes the halting set (meaning there exists a computable coloring of $3$-subsets of $\NN$ such that any of its solutions computes the halting set), none of these principles does the same for any dimension $n\geq 3$. This surprising result is due to Wang \cite{wang2014some}. 

In the present paper we consider generalizations of the Free Set, Thin Set and Rainbow Ramsey theorems to colorings of objects of unbounded dimension. More precisely, we focus on the natural extensions of these principles to colorings of so-called exactly $\omega$-large sets, i.e., finite sets $s$ of natural numbers such that $|s|=1+\min s$.\footnote{The inessential variant with $|s|=\min s$ is also common in the literature.} 

The concept of $\omega$-large (or relatively large) finite subset of the natural numbers is well-known in the proof theory of Arithmetic, as it is the basic ingredient for the celebrated Paris-Harrington independence result for Peano Arithmetic \cite{Paris1977-PARAMI}. In this context, a finite set $s$ of natural numbers is called relatively large (or $\omega$-large) if $|s| \geq \min s$.

Relatively large sets also naturally arise in Ramsey Theory for purely combinatorial reasons. It is well-known, and easy to prove, that the natural generalization of Ramsey's theorem to finite colorings of {\em all finite sets} is a false principle. This is easily witnessed by coloring according to the parity of the size of the set. The following weakening is also false: for every finite coloring $f$ of the finite subsets of the natural numbers there exists an infinite set $H$ of natural numbers such that for infinitely many $n$, $f$ is constant on the $n$-subsets of $H$. Interestingly, a counterexample is given by the coloring that assigns one color to all relatively large sets and the opposite color to all other sets.

While $\omega$-large sets provide a counterexample to the natural extension of Ramsey's theorem to colorings of all finite sets, exactly $\omega$-large sets provide a way to obtain true versions of Ramsey's theorem for colorings of families of finite sets containing elements of unbounded size. Weakening the requirement of homogeneity from all finite sets to all exactly $\omega$-large sets results in a true principle, sometimes called the Large Ramsey Theorem, which we denote by $\RT^{!\omega}$ following \cite{carlucci2014strength}. This principle is arguably the simplest example of a true version of Ramsey's theorem for colorings of objects of {\em unbounded dimensions}, whereas, as noted above, Ramsey's theorem fails for all finite subsets of natural numbers. $\RT^{!\omega}$ is also the base case of a far-reaching generalization of Ramsey's theorem due to Nash-Williams, which ensures monochromatic sets for every finite coloring of families of finite subsets of the natural numbers satisfying some specific properties and called {\em barriers} \cite{Todorcevic+2010}. In this context the family of exactly $\omega$-large sets is known as the Schreier barrier \cite{Todorcevic+2010}.
The Large Ramsey Theorem has been studied from the perspective of Computability Theory and Reverse Mathematics by Carlucci and Zdanowski \cite{carlucci2014strength}, and its generalization to barriers in Computability Theory by Clote~\cite{clote1984recursion}. They proved that it is computationally and proof-theoretically stronger than the usual Ramsey's theorem for each fixed finite dimension ($\RT^n_k$) and even stronger than Ramsey's theorem for all finite dimensions ($\forall n \RT^n_k$). From the point of view of Computability, the theorem corresponds to the $\omega$-th Turing jump $\emptyset^{(\omega)}$; in Reverse Mathematics terms, it is equivalent to $\ACA_0^+$ over $\RCA_0$ (where $\ACA_0^+$ extends $\RCA_0$ by the axiom of closure under the $\omega$-th Turing jump).

In the present paper we formulate and prove the natural generalizations of the Free Set, Thin Set and Rainbow Ramsey theorems to colorings of exactly $\omega$-large sets and we investigate their effective and logical strength. 

\subsection{Framework}
We shall study our statements using two frameworks: Reverse Mathematics and Weihrauch analysis.

\emph{Reverse Mathematics} is a foundational program whose goal is to find optimal axioms to prove ordinary theorems. It uses the framework of subsystems of second-order arithmetic, with a base theory, $\RCA_0$ (Recursive Comprehension Axiom), capturing \emph{computable mathematics}. More precisely, $\RCA_0$ consists of the axioms of Robinson arithmetic, together with the $\Sigma^0_1$-induction scheme and the $\Delta^0_1$-comprehension scheme.
The \emph{$\Sigma^0_1$-induction scheme} states, for every $\Sigma^0_1$-formula $\varphi(x)$,
$$
[\varphi(0) \wedge \forall x (\varphi(x) \rightarrow \varphi(x+1))] \rightarrow \forall y \varphi(y).
$$
The \emph{$\Delta^0_1$-comprehension scheme} states, for every $\Sigma^0_1$-formula $\varphi(x)$ and every $\Pi^0_1$-formula $\psi(x)$,
$$
[\forall x(\varphi(x) \leftrightarrow \psi(x))] \rightarrow \exists Y \forall x (x \in Y \leftrightarrow \varphi(x)).
$$
There exist four other subsystems which, together with $\RCA_0$, calibrate the strength of most theorems. These subsystems are known as the \qt{Big Five} (see Mont\'alban~\cite{montalban_open_2011}). Among these, $\ACA_0$ (Arithmetic Comprehension Axiom) extends the axioms of $\RCA_0$ with the comprehension scheme for all arithmetic formulas. Based on the correspondence between computability and definability, it is equivalent to stating the existence of the Turing jump of any set ($\forall X \exists Y (Y = X')$).

We shall also consider two lesser-known stronger variants of $\ACA_0$, namely, $\ACA_0'$ and $\ACA_0^+$. The system $\ACA_0'$ extends $\RCA_0$ with the axiom stating closure under all finite jumps; the system $\ACA_0^+$ extends $\RCA_0$ with the axiom stating the existence of the $\omega$-jump of any set. The $\omega$-jump of a set~$X$ is the set $X^{(\omega)} = \bigoplus_n X^{(n)}$. The system $\ACA_0^+$ is famous for being the best known upper bound to Hindman's theorem.


Some of our results are expressed in terms of reductions. All principles studied in this paper are of the following logical form: $\forall X (I(X) \to \exists Y S(X,Y))$, where $I(X)$ and $S(X,Y)$ are arithmetical formulas and $X$ and $Y$ are set variables. We refer to such theorems as $\forall\exists$-principles. For principles $\Psf$ of this form we call any $X$ that satisfies $I$ an {\em instance} of $\Psf$ and any $Y$ that satisfies $S(X,Y)$ a {\em solution to} $\Psf$ {\em for} $X$.
We will use the following notion of computable reducibility whose variants became of central interest in Computability Theory and Reverse Mathematics in recent years (see \cite{Dza-Mum:22} for background and motivation).

\begin{definition}
    $\Qsf$ is {\em strongly Weihrauch reducible} to $\Psf$ (denoted $\Qsf \leq_{\sW} \Psf$) if there exist Turing functionals $\Phi$ and $\Psi$ such that for every instance $X$ of $\Qsf$ we have that $\Phi(X)$ is an instance of $\Psf$, and if $Y$ is a solution to $\Psf$ for $\Phi(X)$ then $\Psi(Y)$ is a solution to $\Qsf$ for $X$.
\end{definition}

\subsection{Organization of the paper}

In \Cref{sec:fs_ts_rainbow} we recall the usual, finite-dimensional versions of the Free Set, Thin Set and Rainbow Ramsey theorems and we observe that these principles fail when generalized to all finite sets. For this, we use the notion of exactly $\omega$-large set. Then, in \Cref{sec:fs_ts_rainbow_from_rt}, we recall known facts about the extension of Ramsey's theorem to colorings of exactly $\omega$-large sets and show that the generalized Free Set, Thin Set and Rainbow Ramsey theorems are consequences of the corresponding generalization of Ramsey's theorem. This is analogous to the finite-dimensional case, but some of the proofs require non-trivial adaptations. 
In \Cref{sec:ts_fs_lowerbd} we show that, contrary to their finite-dimensional counterparts which do not code any non-computable set and do not imply $\ACA_0$, the Free Set and Thin Set theorems for exactly $\omega$-large sets code $\emptyset^{(\omega)}$ and imply $\ACA_0^+$. In \Cref{sec:coloring-barriers}, we generalize the statements to a more robust version in terms of barriers, in order to prove upper bounds on a larger class of instances. In \Cref{sec:cone_avoid}, we prove a cone avoidance result for the Rainbow Ramsey Theorem for exactly $\omega$-large sets and, more generally, for colorings of a larger class of barriers of order type $\omega^\omega$. This result entails that none of these principles code the halting set or imply $\ACA_0$. This difference of behavior at the exactly $\omega$-large level is surprising, given the equivalence between the statements $\bigcup_{n \in \NN}\RRT^n$ and $\bigcup_{n \in \NN} \FS^n$ in Reverse Mathematics (see \cite{patey2015somewhere,wang2014some}). 
Last, in \Cref{sec:conclusion}, we conclude and open the discussion to future research directions.

\section{Free sets, thin sets and rainbows}\label{sec:fs_ts_rainbow}

The main goal of this section is to recall the definitions of the usual finite-dimensional Free Set, Thin Set and Rainbow Ramsey Theorem and motivate their generalizations to colorings of exactly $\omega$-large sets in analogy to the case of Ramsey's theorem. 

Let us first fix some notation and recall the definition of these principles for colorings of finite subsets of fixed size. We denote by $\NN$ the set of natural numbers and by $\NN^+$ the set of positive natural numbers. For $X\subseteq \NN$ and $n \in \NN$
we denote by $[X]^{n}$ the set of all $n$-subsets of $X$. We denote by $[X]^{<\omega}$ the set of all finite subsets of $X$. We always assume that sets are presented in increasing order and make no distinction between a set and the sequence of its elements in increasing order. We write $x \cdot y$ for the concatenation of two sequences $x$ and $y$. We identify a natural number with the set of its predecessors, so that, if $k\geq 1$ we can write $f: X \to k$ to declare a function from $X$ to $\{0, 1, \dots, k-1\}$. A function of this type is often called a coloring of $X$ in $k$ colors.

We start by recalling the classical Ramsey's theorem for finite colorings of the $n$-subsets of a countable set. 

\begin{definition}[Ramsey Theorem] Let $n, k\in \NN^+$.  
For every coloring $f:[\NN]^{n} \to k$ there exists an infinite set $H\subseteq \NN$ such that $|f([H]^n)| = 1$. The set $H$ is called homogeneous (or monochromatic) for $f$.
We abbreviate this statement by $\RT^n_k$.
\end{definition}

Ramsey's theorem was first studied in Computability Theory by Jockusch~\cite{JockuschRamsey_1972}, who proved that every computable instance admits an arithmetical solution, and constructed a computable instance of $\RT^3_2$ such that every solution computes the halting set. The formalization of Jockusch's proofs in Reverse Mathematics by Simpson~\cite{SIM:SOSOA} yields that $\RT^1_k$ is provable over~$\RCA_0$ for any standard~$k \in \NN^+$, and $\RT^n_k$ is equivalent to $\ACA_0$ over~$\RCA_0$ for $n \geq 3$ and any $k\geq 2$. The case $n = 2$ was a long-standing open question, until Seetapun~\cite{seetapun1995strength} proved that no computable instance of~$\RT^2_k$ codes the halting set, hence that $\RT^2_k$ is strictly weaker than $\ACA_0$ over~$\RCA_0$.

The following Free Set Theorem, and its weakening the Thin Set theorem, have been introduced by Harvey Friedman \cite{Fri-Sim:00} and first studied in \cite{Cholak_Giusto_Hirst_Jockusch_2005}.

\begin{definition}[Free Set Theorem] Let $n\in\NN^+$.  
For every coloring $f:[\NN]^{n} \to \NN$, there exists an infinite set $H\subseteq \NN$ such that for every set $s\in [H]^{n}$, if $f(s) \in H$ then $f(s) \in s$. The set $H$ is called \emph{free} for $f$ (or $f$-free). We abbreviate this statement by $\FS^n$.
\end{definition}

The notion of free set in Friedman's Free Set Theorem is the same as the one used in a combinatorial characterization theorem for the $\aleph_n$ cardinals by Kuratowski~\cite{Kur:51}. Here, the following property of an $n+1$-subset $U$ of a set $X$ with respect to a function $f:[X]^n \to [X]^{<\omega}$ is considered: for all $x\in U$ we have $x \notin f(U\setminus\{x\})$. For the particular case of a function mapping in (singletons from) $X$ this means that for all $x\in U$ we have $f(U\setminus \{x\}) \neq x$. This is the same as asking that for any $n$-subset $s$ of $U$, if $f(s) \in U$ then $f(s) \in s$, which means that $U$ is free for $f$. Interestingly, in~\Cref{sec:cone_avoid} we are lead to consider extensions of the Free Set Theorem with colors in $[\NN]^{<\omega}$. 

The next theorem is a weakening of the Free Set Theorem. 

\begin{definition}[Thin Set Theorem] For $n\in \NN^+$. 
For every coloring $f:[\NN]^n \to \NN$, there exists an infinite set $H\subseteq \NN$ such that
$f([H]^n) \neq \NN$. The set~$H$ is called \emph{thin} for~$f$ (or $f$-thin). We abbreviate this statement by $\TS^n$.
\end{definition}

The notion of free set might seem ad-hoc at first sight, but can be better understood in the light of the Thin Set Theorem. Indeed, an infinite set $H \subseteq \NN$ is $f$-free if and only if, for every $x \in H$, the set $H \setminus \{x\}$ is $f$-thin with witness color~$x$ (in the sense that $f$ never takes value $x$ on $[H\setminus \{x\}]^n$). Thus, the Free Set Theorem is a natural generalization of the Thin Set Theorem.
The next principle is sometimes called an anti-Ramsey theorem or the Rainbow Ramsey Theorem.
We first need the following definition.

\begin{definition}[$k$-bounded function]
 Let $k\in\NN^+$ and $X$ be a set. A function $f : X \to \NN$ is $k$-bounded if for all $i\in \NN$, 
 $|f^{-1}(i)|\leq k$.
\end{definition}

\begin{definition}[Rainbow Ramsey Theorem]
Let $n, k\in \NN^+$. For all $k$-bounded colorings $f:[\NN]^{n} \to \NN$, there exists an infinite set $H\subseteq \NN$ such that $f$ is injective on $[H]^{n}$. The set~$H$ is called a \emph{rainbow} for $f$ (or an $f$-rainbow).
We abbreviate this statement by $\RRT^{n}_k$. 
\end{definition}

The principles $\TS^n, \FS^n$ and $\RRT^n_k$ have been thoroughly investigated from the point of view of Computability Theory and Reverse Mathematics by a number of authors (\cite{Cholak_Giusto_Hirst_Jockusch_2005, Pat:16, cholak2020thin, Liu2022-LIUTRM}). The general picture that emerged is that these principles are computationally and combinatorially very weak consequences of Ramsey's theorem ($\RT^n \to \FS^n \to \TS^n \wedge \RRT^n$): while $\RT^3$ codes the jump and implies $\ACA_0$, Wang~\cite{wang2014some} showed that the Free Set, Thin Set and Rainbow Ramsey theorems do not code the halting set and satisfy the so-called strong cone avoidance property, which entails that they do not imply $\ACA_0$. 

In the remainder of this section we observe that, similarly to the case of Ramsey's theorem, the natural generalizations of the Free Set, Thin Set and Rainbow Ramsey theorems to colorings of all finite subsets of $\NN$ {\em fail}. 

Let $\FS^{<\omega}$ be the following principle: For every coloring $f:[\NN]^{<\omega} \to \NN$ there exists an infinite free set. Analogously, let $\TS^{<\omega}$ be the following principle: for every coloring $f:[\NN]^{<\omega} \to \NN$ there exists an infinite thin set. 
Finally, let $\RRT^{<\omega}_k$ be the following statement: For all $k$-bounded coloring $f:[\NN]^{<\omega} \to \NN$ there exists
an infinite set $H\subseteq \NN$ such that $f$ is injective on $[H]^{<\omega}$ (i.e., $H$ is a rainbow for $f$).

\begin{proposition}
There exists a coloring of the finite subsets of the natural numbers that admits no infinite thin set (i.e., $\TS^{<\omega}$ is false).
\end{proposition}

\begin{proof}
Let $f:[\NN]^{<\omega} \to \NN$ be defined by setting $f(s) = |s|$. 
Let $X$ be an infinite subset of $\NN$. Then $f([X]^{<\omega})= \NN$.
\end{proof}

Similarly to the finite dimensions case~\cite{Cholak_Giusto_Hirst_Jockusch_2005}, the Free Set Theorem for all finite sets implies the Thin Set Theorem for all finite sets. 
We formulate this fact in terms of reductions since the exact same argument applies to other principles of interest in this paper. The observation that the proof of the next proposition (see proof Theorem 3.2 in \cite{Cholak_Giusto_Hirst_Jockusch_2005}) applies to barriers was one of the starting points of the present work. 

\begin{proposition}\label{prop:ts_red_fs}
$\TS^{<\omega}\leq_{\sW} \FS^{<\omega} $.
\end{proposition}

\begin{proof}
Let $f:[\NN]^{<\omega}\to \NN$. Let $A$ be an infinite free set for $f$. 
Let $B$ be a non-empty subset of $A$ such that $A\setminus B$ is infinite. 
We claim that $A\setminus B$ is thin for $f$. Assume, by way of contradiction, that for all $n \in \NN$ there exists an $s_n\in [A\setminus B]^{<\omega}$ such that $f(s_n) = n$. Take $n\in B$. Thus, $n \in A$. Since $A$ is free for $f$, it must be the case that $n \in s_n$, contradicting the fact that $s_n$ was chosen 
in $A\setminus B$. 
\end{proof}

As a corollary we obtain the following proposition. 

\begin{proposition}
There exists a coloring of the finite subsets of the natural numbers that admits no infinite free set (i.e., $\FS^{<\omega}$ is false). 
\end{proposition}

We next show that the natural generalization of the Rainbow Ramsey Theorem to colorings of all finite subsets of the natural numbers fails.
The proof features the notion of exactly $\omega$-large set, which is central for the present paper. 

\begin{definition}[(Exactly) $\omega$-large sets] 
A finite $s\subseteq \NN$ is $\omega$-large if $|s| \geq 1+\min s$ and is exactly $\omega$-large if $|s| = 1+\min s$.
\end{definition}

The $\omega$-large sets are also known as {\em relatively large} sets in the literature. For an infinite $X\subseteq \NN$ we denote by $[X]^{!\omega}$ the family of all exactly $\omega$-large subsets of $X$.
The family $[\NN]^{!\omega}$ coincides with the famous Schreier barrier used in better quasi ordering theory and Banach Space Theory~\cite{Todorcevic+2010}.

\begin{proposition}\label[proposition]{prop:rrt-full-is-false}
There exists a $2$-bounded coloring of the finite subsets of the natural numbers that admits no infinite rainbow (i.e., $\RRT^{<\omega}_2$ is false).
\end{proposition}

\begin{proof}
A finite set $s \subseteq \NN$ is called {\em quasi-exactly $\omega$-large} if there exists an exactly $\omega$-large set $t$ such that
$s= t \setminus \{\min t\}$. Obviously if $t$ is exactly $\omega$-large then $t\setminus \{\min t\}$ is quasi exactly $\omega$-large. Also it
is easy to see that if $s$ is quasi exactly $\omega$-large then there exists a unique exactly $\omega$-large $t$ such that $s = t \setminus \{\min t\}$.
Let $b:[\NN]^{<\omega} \to \NN$ be a bijection. We define a $2$-bounded coloring $f:[\NN]^{<\omega} \to \NN$ as follows. 
If $s$ is neither exactly $\omega$-large nor quasi exactly $\omega$-large then $f(s) = b(s)$. If $s$ is exactly $\omega$-large then $f(s) = f(s\setminus \{\min s\}) = b(s).$
Let $H\subseteq \NN$ be infinite. Then for any $s \in [H]^{!\omega}$ we have $s\setminus \{\min s\} \in [H]^{<\omega}$. 
But for any such $s$ we have $f(s) = f(s\setminus \{\min s\})$. Thus $H$ is not a rainbow for $f$. 

\end{proof}

\section{Colorings of exactly $\omega$-large sets}\label{sec:fs_ts_rainbow_from_rt}

In this section, we introduce and prove the generalizations of the Free Set, Thin Set and Rainbow Ramsey theorems to colorings of exactly~$\omega$-large sets. We also establish some basic implications and relations among those principles.

As mentioned, Ramsey's theorem fails when generalized to the family of all finite sets. On the other hand, there exists a natural generalization of Ramsey's theorem to some families of finite sets of {\em unbounded} size that is central for our investigation. The following is the generalization of Ramsey's theorem to colorings of exactly $\omega$-large sets, as a particular case of more general theorems by Pudl\'ak and R\"{o}dl~\cite{Pud-Rod:82} and by Farmaki and Negrepontis~\cite{Far-Neg:08}. 
Besides, it is a particular case of Nash-Williams' generalization of Ramsey's theorem to barriers, which in turn is a consequence of the Clopen Ramsey Theorem (see~\cite{SIM:SOSOA}). 





\begin{definition}[Large Ramsey Theorem]\label{def:largeRT} Let $k\in \NN^+$. 
For every coloring $f:[\NN]^{!\omega} \to k$, there exists an infinite set $H \subseteq \NN$ such that $|f([H]^{!\omega})| = 1$. We abbreviate this statement by  $\RT^{!\omega}_k$.
\end{definition}

The classical Ramsey Theorem is a statement about cardinality, in the precise sense that $\RT^n_k$ proves over $\RCA_0$ the following stronger statement \qt{For every infinite set $X \subseteq \NN$ and every coloring $f : [X]^n \to k$, there is an infinite $f$-homogeneous set~$H \subseteq X$.} Proposition 8.3.4 of \cite{Dza-Mum:22} gives the equivalence of this general version of Ramsey's Theorem -- denoted $\mathsf{General}$-$\RT^n_k$ in \cite{Dza-Mum:22} -- with $\RT^n_k$ both in terms of Weihrauch reductions and in terms of provability over $\RCA_0$. The situation is more complex in the case of Large Ramsey Theorem, as there is no clear direct equivalence between $\RT^{!\omega}_k$ and its general version ensuring homogeneous sets in any countable subset of $\NN$.
Nevertheless, the classical proof and the computability-theoretic bounds of Large Ramsey Theorem as defined in Definition \ref{def:largeRT} hold for the general version of the statement. Because of this, Large Ramsey Theorem can arguably be considered as a non-robust statement. A first solution, adopted by Carlucci and Zdanowski~\cite{carlucci2014strength}, consisted in directly studying the theorem in its general formulation, where an instance is a pair $(X, f)$ with $X$ an infinite subset of $\NN$ and $f$ a $k$-coloring of the exactly large subsets of $X$. However, making the domain part of the instance raises some issues when considering strong cone avoidance as we do in~\Cref{sec:cone_avoid} for the Large Ramsey Theorem, since the domain can be chosen to be sparse enough to compute any hyperarithmetic set. We shall therefore adopt a different approach, and prove our lower bounds in terms of the weaker versions of our statements, while proving the cone avoidance results on a generalized version formulated in terms of barriers, that will be presented in~\Cref{sec:cone_avoid}. All these versions are equivalent over~$\RCA_0$ to $\ACA_0^+$, so either formulation can be chosen.
The following theorem summarizes the known bounds on $\RT^{!\omega}_k$. 

\begin{theorem}\label{thm:car-zda}
\leavevmode
\begin{enumerate}
\item All computable finite colorings of  $[\NN]^{!\omega}$ admit an infinite monochromatic set computable in $\emptyset^{(\omega)}$. 
\item There exists a computable coloring of $[\NN]^{!\omega}$ in $2$ colors such that all infinite monochromatic sets compute 
$\emptyset^{(\omega)}$. 
\item $\RT^{!\omega}_2$ is equivalent to $\ACA_0^+$ over $\RCA_0$.
\end{enumerate}
\end{theorem}
\begin{proof}
Carlucci and Zdanowski \cite{carlucci2014strength} proved the results for the general version of $\RT^{!\omega}$. The lower bounds in points 2.~and 3.~for our weaker formulation follow from Theorem 4 in \cite{car-main-zda2024} as well as from Corollary \ref{cor:RTtoFS}, Corollary \ref{fs-omega-implies-acaplus} and Corollary \ref{fs-omega-computes-omega-jump} of the present paper. Clote~\cite{clote1986generalization} contains a proof of point 1.~for a larger family of colorings and point 2.~for colorings of a family closely related to $[\NN]^{!\omega}$ (see Theorem~\ref{thm:Clote} below).
\end{proof}

It is quite natural to ask if the natural generalizations of the Free Set, Thin Set and Rainbow Ramsey theorems to colorings of exactly~$\omega$-large sets hold. 
The following is the natural generalization of the Free Set Theorem to colorings of exactly $\omega$-large sets.

\begin{definition}[Large Free Set Theorem]\label{def:large_free}
For every coloring $f:[\NN]^{!\omega} \to \NN$ there exists an infinite set $H\subseteq\NN$ such that for every set $s \in [H]^{!\omega}$, $f(s) \notin (H \setminus s)$. The set $H$ is called \emph{free} for $f$. We abbreviate this statement by $\FS^{!\omega}$. 
\end{definition}


There exists a direct combinatorial proof of $\FS^{!\omega}$, as for the proof of $\RT^{!\omega}_k$, involving countable applications of $\RT^n_k$ and $\FS^n$ for $n \in \NN^+$ and a final application of $\FS^1$. 
A computability-theoretic analysis of this proof yields a solution computable in the $\omega$-jump of the instance.
We rather establish $\FS^{!\omega}$ by reduction to $\RT^{!\omega}$. The proof combines ideas from the proof of $\FS^n$ from $\RT^n_{2n+2}$ (Theorem 5.2 and Corollary 5.3 in \cite{Cholak_Giusto_Hirst_Jockusch_2005}) and from the proof of Theorem 4.1 in \cite{carlucci2014strength}.

\begin{theorem}\label{thm:fs_red_rt}
    $\FS^{!\omega} \leq_{\sW} \RT^{!\omega}_2$.
\end{theorem}

\begin{proof}
For $s = \{s_0, \dots, s_{s_0}\}$ an exactly $\omega$-large set with $s_0 > 0$, and $s_0 < \dots < s_{s_0}$ define $s \ominus 1$ to be the following exactly $\omega$-large set: $\{s_0 - 1,s_1 - 1, \dots, s_{s_0 - 1} - 1\}$. The idea is to keep one degree of freedom: $s_{s_0}$, as in the proof of Carlucci and Zdanowski~\cite[Proposition 4.1]{carlucci2014strength}.

Let $f : [\NN]^{!\omega} \to \NN$ be an instance of $\FS^{!\omega}$. Consider the following function $g : [\NN]^{!\omega} \to 2$ defined by induction. Note that the function $g$ is called recursively on  lexicographically smaller parameters, so the induction is well-defined since the lexicographic order is a well-order. For $s = \{s_0, \dots, s_{s_0}\}$ an exactly $\omega$-large set:


\begin{equation*}
    g(s) = \begin{cases}
           0 \hspace{1.5cm} \text{if } f(s \ominus 1) = s_i - 1 \textit{ for some } i < s_0 & \\
           1 - g(f(s \ominus 1) + 1, s_1, \dots, s_{f(s \ominus 1) + 1}) & \\ 
            \hspace{1.7cm} \text{if } f(s \ominus 1) < s_0 - 1  \\
           1 - g(s_0, \dots, s_i, f(s \ominus 1) + 1, s_{i+2}, \dots, s_{s_0}) & \\
            \hspace{1.7cm} \text{if } f(s \ominus 1) \in (s_i - 1, s_{i+1} - 1) \textit{ for some } i < s_0 - 1  \\
           0 \hspace{1.5cm} \text{if } f(s \ominus 1) \in (s_{s_0 - 1} - 1, s_{s_0} - 1) \\
           1 \hspace{1.5cm}  \text{otherwise } (\textit{if } f(s \ominus 1) \geq s_{s_0} ) \\      
           \end{cases} \quad
\end{equation*}

Let $H = \{x_0, x_1, \dots \}\subseteq \NN^+$ 
be an infinite $g$-homogeneous set as given by $\RT^{!\omega}_2$. We claim that $H' = \{x_0 - 1, x_1 - 1, \dots \}$ is $f$-free. Consider some exactly $\omega$-large set $\{s_0 - 1, \dots, s_{s_0 - 1} - 1\} \subseteq H'$. Then $\{s_0, \dots, s_{s_0 - 1}\} \subseteq H$. There are two cases:

\textbf{Case 1:} $H$ is homogeneous for the color $0$. Then, take $s_{s_0}$ to be the next element of $H$ after $s_{s_0 - 1}$ and write $s = \{s_0, \dots, s_{s_0}\}$, then $g(s) = 0$. There are four subcases:

\textbf{Subcase 1.1:} $f(s \ominus 1) = s_i - 1$ for some $i < s_0$. In that case we are done.

\textbf{Subcase 1.2:} $f(s \ominus 1) \in (s_{s_0 - 1} - 1, s_{s_0} - 1)$. In that case, by definition of $s_{s_0}$, $f(s \ominus 1)$ is not in $H'$.

\textbf{Subcase 1.3:} $f(s \ominus 1) < s_0 - 1$ and $g(f(s \ominus 1) + 1, s_1, \dots, s_{f(s \ominus 1) + 1}) = 1$. If $f(s \ominus 1) \in H'$ this contradicts the fact that $H$ is $g$-homogeneous for the color $0$. 

\textbf{Subcase 1.4:} $f(s \ominus 1) \in (s_i - 1, s_{i+1} - 1)$ for some $i < s_0 - 1$ and $g(s_0, \dots, s_i, f(s \ominus 1) + 1 ,s_{i+2}, \dots, s_{s_0}) = 1$. If $f(s \ominus 1) \in H'$ this contradicts the fact that $H$ is $g$-homogeneous for the color $0$. 

\textbf{Case 2:} $H$ is homogeneous for the color $1$. Take $s_{s_0} \in H$ bigger than $s_{s_0 - 1}$ and write $s = \{s_0, \dots, s_{s_0}\}$, then $g(s) = 1$. There are three subcases:

\textbf{Subcase 2.1:} $f(s \ominus 1) \geq s_{s_0}$. This case is impossible, indeed, as $H$ is infinite, there exists some element $x \in H$ such that $x > f(s \ominus 1) + 1$ and therefore $f(s \ominus 1) \in (s_{s_0} - 1, x - 1)$, which leads to $g(s_0, \dots, s_{s_0 - 1}, x) = 0$ contradicting the fact that $H$ is $g$-homogeneous for the color $1$.

\textbf{Subcase 2.2:} $f(s \ominus 1) < s_0 - 1$ and $g(f(s \ominus 1) + 1, s_1, \dots, s_{f(s \ominus 1) + 1}) = 0$. If $f(s \ominus 1) \in H'$ this contradicts the fact that $H$ is $g$-homogeneous for the color $1$. 

\textbf{Subcase 2.3:}  $f(s \ominus 1) \in (s_i - 1, s_{i+1} - 1)$ for some $i < s_0 - 1$ and $g(s_0, \dots, s_i, f(s \ominus 1) + 1 ,s_{i+2}, \dots, s_{s_0}) = 0$. If $f(s \ominus 1) \in H'$ this contradicts the fact that $H$ is $g$-homogeneous for the color $1$.  \\

Notice that $g$ is uniformly computable in $f$ and that $H'$ is uniformly computable in $H$. Therefore, $\FS^{!\omega} \leq_{\sW} \RT^{!\omega}_2$.

\end{proof}

The above proof uses the fact that the exactly~$\omega$-large sets are well-ordered under lexicographic ordering. Since the order type of this ordering is $\omega^\omega$ and the statement \qt{$\omega^\omega$ is well-ordered} implies the consistency of~$\RCA_0$ (see \cite{gentzen1967wiederspruchsfreiheit,hajek1998metamathematics}), the above proof is not formalizable in $\RCA_0$. Yet, since $\RT^{!\omega}_2$ implies $\ACA_0$ and the latter proves that $\omega^\omega$ is well-ordered we obtain the following corollary.

\begin{corollary}\label{cor:RTtoFS}
    $\RCA_0 \vdash \RT^{!\omega}_2 \to \FS^{!\omega}$.
\end{corollary}


We next introduce the generalization of the Thin Set Theorem to colorings of exactly $\omega$-large sets. For the fixed-dimension case, Cholak et al.~\cite{Cholak_Giusto_Hirst_Jockusch_2005} have proved in $\RCA_0$ that for all $k\geq 2$, $\FS^k$ implies $\TS^k$ 
(see Theorem 3.2 in~\cite{Cholak_Giusto_Hirst_Jockusch_2005}; the proof yields a strong Weihrauch reduction). A completely analogous argument establishes that the Thin Set Theorem follows from (and is reducible to) the Free Set Theorem for colorings of exactly $\omega$-large sets. 

\begin{definition}[Large Thin Set Theorem]\label{thm:thin}
For every coloring $f:[\NN]^{!\omega} \to \NN$, there exists an infinite set $H\subseteq \NN$ such that $f([H]^{!\omega}) \neq \NN$. The set $H$ is called \emph{thin} for $f$. We abbreviate this statement by $\TS^{!\omega}$.
\end{definition}

\begin{theorem}\label{thm:thin-sw-fs-omega}
$\TS^{!\omega} \leq_{\sW} \FS^{!\omega}$ and $\RCA_0 \vdash \FS^{!\omega}\to \TS^{!\omega}.$
\end{theorem}

\begin{proof} Completely analogous to the proof of Proposition \ref{prop:ts_red_fs}.
\end{proof}


The following proposition states that requiring that one color is omitted is equivalent to requiring that infinitely many colors are omitted. The result is the analogue of Theorem 3.5 in~\cite{Cholak_Giusto_Hirst_Jockusch_2005} and can be proved by exactly the same proof.

\begin{proposition}
    For every coloring $f:[\NN]^{!\omega} \to \NN$ there exists an infinite set $X\subseteq\NN$ such that $\NN\setminus f([X]^{!\omega})$ is infinite. Moreover, the just stated principle is strongly Weihrauch-equivalent to $\TS^{!\omega}$ and provably equivalent to the latter over $\RCA_0$.
\end{proposition}

\begin{proof}
See proof of Theorem~3.5 in~\cite{Cholak_Giusto_Hirst_Jockusch_2005}.
\end{proof}

We now turn to the generalization of the Rainbow Ramsey Theorem. 

\begin{definition}[Large Rainbow Ramsey Theorem]
Let $k\in \NN^+$. For every $k$-bounded coloring $f:[\NN]^{!\omega} \to \NN$ there exists an infinite set $H\subseteq \NN$ such that $f$ is injective on $[H]^{!\omega}$. The set~$H$ is called a \emph{rainbow} for~$f$. We abbreviate this statement by $\RRT^{!\omega}_k$.
\end{definition}

For the fixed dimension case, Galvin gave a reduction to Ramsey's theorem. His argument is formalizable in $\RCA_0$ and yields that for each $n, k\in\NN$, $\RRT^n_k \leq_\sW \RT^n_k$ (see \cite{CsimaMileti_2009}, proof of Theorem 5.2). Csima and Mileti also showed that for each $n,k\in\NN^+$, $\RRT^n_k$ follows from $\RRT^{n+1}_k$ over $\RCA_0$ (the proof of Theorem 5.3 in \cite{CsimaMileti_2009} yields a strong Weihrauch reduction).

We first observe that Galvin's argument adapts to the case of colorings of exactly $\omega$-large sets and establishes a strong Weihrauch reduction.

\begin{theorem}\label{thm:largerrt}
For all $k\in\NN^+$, $\RRT^{!\omega}_k \leq_\sW \RT^{!\omega}_k$. Moreover $\RCA_0 \vdash \forall k (\RT^{!\omega}_k\to \RRT^{\omega}_k)$.
\end{theorem}

\begin{proof}
Let $f:[\NN]^{!\omega} \to \NN$ be $k$-bounded and fix a computable bijection $b(\cdot): [\NN]^{!\omega} \to \NN$.
Define $g:[\NN]^{!\omega} \to k$ as follows:
$$ g(s) = |\{ t \in [\NN]^{!\omega} \,:\, b(t) < b(s) \text{ and } f(s) = f(t)\}|.$$
The fact that $g$ is a $k$-coloring depends on the hypothesis that $f$ is $k$-bounded.
Let $H$ be an infinite set such that $g$ is constant on $[H]^{!\omega}$, as given by $\RT^{!\omega}_k$. 
Let $s, t \in [H]^{!\omega}$. Since $g(s) = g(t)$ and either $b(s) < b(t)$ or $b(t) < b(s)$
we have that $f(s) \neq f(t)$. Thus $H$ is rainbow for $f$. 
\end{proof}

The following is an adaptation of a result by Wang \cite{wang2014some}, who showed that for every $n$, $\RRT^n_2 \leq_\sW \FS^n$.

\begin{proposition}\label{prop:rrt-omega-2-sw-fs-omega}
$\RRT^{!\omega}_2 \leq_\sW \FS^{!\omega}$.
\end{proposition}

\begin{proof}
Fix a computable bijection $b: [\NN]^{!\omega} \to \NN$. Let $f:[\NN]^{!\omega}\to\NN$
be $2$-bounded. Define $g:[\NN]^{!\omega} \to \NN$ as follows:
$$
g(s)=
\begin{cases}
\min(t\setminus s) & \mbox{ if there is a } t\in [\NN]^{!\omega} \text{ such that } b(t)<b(s) \text{ and } f(s)=f(t),\\
0 & \mbox{ otherwise.}
\end{cases}
$$

Since $f$ is $2$-bounded, if $t$ exists in the definition of $g$ then it is unique. If $t$ and $s$ are distinct
exactly $\omega$-large sets then $(t\setminus s)\neq \emptyset$, since $t\subseteq s$ is impossible. Let $A$ be 
an infinite $g$-free set. We claim that $A$ is a rainbow for $f$. Suppose otherwise, by way of contradiction, 
as witnessed by $s, t \in [A]^{!\omega}$ such that $f(s) = f(t)$. Without loss of generality we can assume
$b(t) < b(s)$. Then $g(s) = \min(t\setminus s) \in A \setminus s$, contradicting that $A$ is $g$-free.
\end{proof}


Wang~\cite{wang2014some} proved that the Free Set, Thin Set and Rainbow Ramsey theorems for fixed-sized sets satisfy cone-avoidance. This entails that none of these principles codes the halting set or implies $\ACA_0$.
A natural question is whether the same is true of their versions for exactly $\omega$-large sets.
We will show that $\FS^{!\omega}$ and $\TS^{!\omega}$ code $\emptyset^{(\omega)}$ and imply the much stronger system $\ACA_0^+$, while $\RRT^{!\omega}$ admits cone avoidance.

\bigskip

Some first lower bounds on our principles can be obtained by adapting results from the finite-dimensional case. For example, the Rainbow Ramsey Theorem with internal quantification over all fixed finite dimensions $\forall n \RRT^n_2$ follows from the Rainbow Ramsey Theorem for exactly $\omega$-large sets. 
\begin{proposition}
For each $k, n\in\NN^+$, $\RRT^{!\omega}_k \geq_\sW \RRT^n_k$.
Moreover, $\RCA_0 \vdash \forall k (\RRT^{!\omega}_k \to \forall n \RRT^n_k)$.\footnote{The published proof \cite{carlucci2025ramsey} of that result is flawed; the present version corrects it.}
\end{proposition}

\begin{proof}
Let $n > 0$ and let $f:[\NN]^n \to \NN$ be $k$-bounded. Define $g:[\NN]^{!\omega}\to \NN$ as follows. For an exactly $\omega$-large set $s = \{s_0, \dots, s_{s_0}\}$,

$$
g(s)=
\begin{cases}
\langle 0, s_0, f(s_1, \dots, s_n), s_{n+1}, \dots, s_{s_0} \rangle & \mbox{ if } s_0 \geq n,\\
\langle 1, s \rangle & \mbox{ is } s_0 < n,
\end{cases}
$$
where $\langle \cdot \rangle : \NN^{<\omega} \to \NN$ is a fixed computable bijection.

\textbf{Claim: $g$ is $k$-bounded.} Let $s = \{s_0, \dots, s_{s_0}\}$ be exactly $\omega$-large. If $s_0 < n$ then there are no other exactly $\omega$-large set $t$ such that $g(s) = g(t)$. If $s_0 \geq n$, suppose that some exactly $\omega$-large set $t = \{t_0, \dots, t_{t_0}\}$ satisfy $g(s) = g(t)$. Then, necessarily $t_0 = s_0$, $s_i = t_i$ for $i > n$, and $f(s_1, \dots, s_n) = f(t_1, \dots, t_n)$, hence there are at most $k$ such $t$ by $k$-boundedness of $f$. \\

Now, let $A$ be an infinite rainbow for $g$ and define $B = A \cap [n, \infty)$, then $B \setminus \{\min B\}$ is an infinite rainbow for $f$. Indeed, let $s,t \in [B \setminus \{\min B\}]^n$ such that $f(s) = f(t)$. Since $A$ is infinite, we can choose elements $x_{n+1} < \dots < x_{\min B}$ in $A$ such that $x_{n+1} > \max \{\max s, \max t\}$. Define $s' = \{\min B\} \cup s \cup \{x_{n+1}, \dots, x_{\min B}\}$ and $t' = \{\min B\} \cup t \cup \{x_{n+1}, \dots, x_{\min B}\}$. Then $s'$ and $t'$ are exactly $\omega$-large and $g(s') = g(t')$ by construction. Since $A$ is a rainbow for $g$, this implies $s' = t'$, and therefore $s = t$. Hence $B \setminus \{\min B\}$ is a rainbow for $f$.
\end{proof}

Analogous results can be obtained for the Large Free Set and the Large Thin Set Theorem, so that $\Psf^{!\omega}$ implies $\forall n \Psf^n$ over $\RCA_0$ for $\Psf\in\{\FS,\TS\}$. Both implications are also witnessed by strong Weihrauch reductions. We omit the proofs since these results are superseded by the results of the next section where we prove that $\FS^{!\omega}$ and $\TS^{!\omega}$ imply $\ACA_0^+$.

\section{Large Thin and Free Set Theorems code $\emptyset^{(\omega)}$}\label{sec:ts_fs_lowerbd}

In this section we establish strong lower bounds on $\FS^{!\omega}$ 
and $\TS^{!\omega}$ showing that both these principles code $\emptyset^{(\omega)}$ and imply $\ACA_0^+$. This should be contrasted with the fact that neither $\forall n \FS^n$ nor $\forall n \TS^n$ imply $\ACA_0$.

The first goal is to prove the existence of a computable instance of $\TS^{!\omega}$ such that every solution uniformly computes $\emptyset^{(\omega)}$. In particular, $\TS^{!\omega}$ admits the same lower bound as $\RT^{!\omega}_2$. Since $\TS^{!\omega}$ is strongly Weihrauch reducible to $\FS^{!\omega}$ by \Cref{thm:thin-sw-fs-omega}, it follows that there is a computable instance of $\FS^{!\omega}$ such that every solution computes $\emptyset^{(\omega)}$. 
By results in the previous section, this bound is optimal, since every computable instance of $\FS^{!\omega}$ admits a solution computable in~$\emptyset^{(\omega)}$ by Theorem \ref{thm:fs_red_rt} and Theorem \ref{thm:car-zda}. The following definition of thinness is technically convenient. 

\begin{definition}
Let $C$ be any non-empty set. Given a coloring $f : [\NN]^n \to C$, a set~$H \subseteq \NN$ is \emph{$f$-thin for color~$c \in C$} if $c \not \in f([H]^n)$. A set ~$H \subseteq \NN$ is $f$-thin (or thin for $f$) if $H$ is $f$-thin for some color $c\in C$.
\end{definition}

The definition of $f$-thin set depends on the choice of codomain of the function. It will always be clear from the context. 


The following version of the Thin Set Theorem for finite colorings was introduced in~\cite{dorais2016uniform} and is useful for our purposes.

\begin{definition}[Thin Set Theorem for finite colorings]\label{def:ts-finite}
Let $n, k\in \NN^+$ with $k\geq 2$. For every coloring $f : [\NN]^n \to k$, there exists an infinite set $H \subseteq \NN$ such that $H$ is thin for $f$. We abbreviate this statement by $\TS^n_{k}$.
\end{definition}

Dorais et al.~\cite[Proposition 5.5]{dorais2016uniform} proved the existence, for every~$n \in \NN^+$, of a computable coloring $f : [\NN]^{n+2} \to 2^n$ such that every infinite $f$-thin set computes $\emptyset'$. However, their proof is not uniform, which is a required feature for our construction to code $\emptyset^{(\omega)}$. We prove the existence, for every~$n \geq 2$, of a computable coloring $f : [\NN]^{n+1} \to n$ such that every infinite $f$-thin set uniformly computes~$\emptyset'$.


\begin{lemma}\label{lem:arrays}
    There exists two computable arrays $(e_{n,k})_{n,k \in \NN}$ and $(d_{n,k})_{n,k \in \NN}$ of Turing indexes such that for every $n, k \in \NN^+$ with $n \geq 2$, $\Phi_{e_{n,k}}^{\emptyset^{(k)}}$ is a coloring $f_n^k : [\NN]^{n} \to n$ such that for every infinite $f_n^k$-thin set $H$, $\Phi_{d_{n,k}}^{H} = \emptyset^{(k)}$.
\end{lemma}

\begin{proof}
Let $k\in\NN^+$. Let $g_k$ be a uniform modulus of the set $\emptyset^{(k)}$. Note that $\Delta^0_{k+1}$-indexes for each function $g_k$ can be found computably, uniformly in~$k$. 

For $n \geq 2$ let $f_n^k(x_0, \dots,x_{n-1}) = n - 1$ if $g_k(x_0) \leq x_1$ and if $g_k(x_0) > x_1$, let $f_n^k(x_0, \dots,x_{n-1}) = i$ for $i < n - 1$ the largest value such that $g_k(x_0) > x_{i+1}$.

Let $H$ be an infinite $f_n^k$-thin set for some color $c$. The color $c$ cannot be $n-1$ as for a given $x_0 \in H$ there exists $x_1 <\dots < x_{n-1} \in H \setminus \{0, \dots, g_k(x_0)\}$, hence $f^k_n(x_0, \dots, x_{n-1}) = n-1$.

We claim $H$ is thin for color $n-2$. Indeed, assume by contradiction that there exists some tuple $x_0 < \dots < x_{n-1} \in H$ such that $g_k(x_0) > x_{n-1}$ so that $f^k_n$ takes color $n-2$ on $\{x_0, \dots, x_{n-1}\}$. Since $c < n-2$ and $H$ is infinite, there exists some $y_{c+2}< \dots < y_{n-1} \in H \setminus \{0, \dots ,g_k(x_0)\}$, therefore $f^k_n(x_0, \dots, x_{c+1},y_{c+2}, \dots, y_{n-1}) = c$, contradicting our assumption that $H$ is $f_n^k$-thin for $c$.

Write $H = \{x_0 < x_1 < \dots \}$. For every $i \in \NN$, $g_k(x_i) \leq x_{i + n - 1}$, hence $H$ computes a function dominating $g_k$ and therefore computes $\emptyset^{(k)}$. This computation can be done uniformly in the set $H$, $k$ and $n$. 
\end{proof}


\begin{lemma}
    There exists a functional $\Gamma$ such that for every set $X$ and every index $e$, if $\Phi_e^{X'}$ is a coloring $f : [\NN]^n \to \ell$, then $\Gamma^X(e)$ is a coloring $g : [\NN]^{n+1} \to \ell$ such that if an infinite set $H$ is $g$-thin for a color $c$, then $H$ is $f$-thin for the same color $c$. 
\end{lemma}

\begin{proof}
Consider a set $X$ and an index $e$ such that $\Phi_e^{X'}$ is a coloring $f : [\NN]^n \to \ell$ for some $n, \ell \in \NN^+$. Let $(f_s)_{s\in \NN}$ be a $\Delta_2^0(X)$-approximation of $f$ (indexes for such an approximation, where the $f_s$ are seen as $X$-computable functions, can computably be found uniformly in $X$ and $e$). Finally, consider the following $X$-computable coloring $g$ defined by $g(x_0, \dots, x_{n}) = f_{x_n}(x_0, \dots, x_{n-1})$ (again, the construction is uniform).

Let $H \subseteq \NN$ be an infinite set. If for some $x_0 < \dots < x_{n-1} \in H$ we have $f(x_0, \dots, x_{n-1}) = c$ for some color $c$, then $f_s(x_0, \dots, x_{n-1}) = c$ for every $s$ bigger than a certain threshold. Thus, as $H$ is infinite, there exists some $x_n \in H$ such that $g(x_0, \dots, x_n) = c$.  
\end{proof}

Combining these two lemmas, we get the following, where the array $(d_{n,k})_{n,k\in\NN}$ is as in~\Cref{lem:arrays}.

\begin{lemma}\label[lemma]{lem:uniform-computable-thin-coloring}
    There exists a uniformly computable sequence of colorings $f_{n,k} : [\NN]^{n+k} \to n$, for $n, k \in \NN^+$, such that for every infinite $f_{n,k}$-thin set $H$, $\Phi_{d_{n,k}}^H = \emptyset^{(k)}$.
\end{lemma}

The following theorem states that $\TS^{!\omega}$ does not admit cone avoidance in a strong sense: there exists a single computable instance of $\TS^{!\omega}$ that computes $\emptyset^{(\omega)}$.

\begin{theorem}\label[theorem]{ts-omega-computes-omega-jump}
    There exists a computable coloring $f : [\NN]^{!\omega} \to \NN$ such that every infinite $f$-thin set~$H$ computes $\emptyset^{(\omega)}$. Moreover, this computation is uniform in~$H$ and an avoided color~$c$.
\end{theorem}

\begin{proof}
Let $f : [\NN]^{!\omega} \to \NN$ be defined as follows (where the functions $f_{n,k}$ are from \Cref{lem:uniform-computable-thin-coloring}): for $x_0 < \dots < x_{x_0}$ in $\NN$ with $x_0 > 1$, let $f(x_0, \dots, x_{x_0})= f_{x_0 - k, k}(x_1, \dots, x_{x_0})$ for $k = \lceil \frac{x_0}{2} \rceil$; for $s \in [\NN]^{!\omega}$ wit $\min s \leq 1$ set $f(s) = 0$. By \Cref{lem:uniform-computable-thin-coloring}, $f$ is indeed computable.

Let $H$ be an infinite $f$-thin set for some color $c$ and let $n \in \NN$. Since $H$ is infinite, there exists some $x_0 \in H$ such that, by letting $k = \lceil \frac{x_0}{2} \rceil$, we have $x_0 - k > c$ and $k \geq n$. The set $G = H \setminus \{0, \dots, x_0\}$ is infinite and $f$-thin for $c$. By definition of $f$, $G$ is $f_{x_0 - k, k}$-thin for the color~$c$. Note that $c$ is part of the range of $f_{x_0 - k, k}$ as $x_0 - k > c$. By \Cref{lem:uniform-computable-thin-coloring}, since $k \geq n$, $G \geq_T \emptyset^{(n)}$. Since $G$ is obviously $H$-computable, we have $H\geq_T \emptyset^{(n)}$.
This computation can be done uniformly in~$H$, $c$ and $n$, thus $H \geq_T \emptyset^{(\omega)}$ uniformly in $H$ and~$c$.
\end{proof}

\begin{corollary}\label[corollary]{fs-omega-computes-omega-jump}
There exists a computable coloring $f : [\NN]^{!\omega} \to \NN$ such that every infinite $f$-free set computes $\emptyset^{(\omega)}$.
\end{corollary}
\begin{proof}
Immediate by \Cref{ts-omega-computes-omega-jump} and the fact that $\TS^{!\omega} \leq_{\sW} \FS^{!\omega}$ by \Cref{thm:thin-sw-fs-omega}.
\end{proof}

Since the proof of \Cref{ts-omega-computes-omega-jump} relativizes and is formalizable in $\RCA_0$ we obtain the following Reverse Mathematics corollary.

\begin{corollary}\label[corollary]{fs-omega-implies-acaplus}
    Each of $\TS^{!\omega}$ and $\FS^{!\omega}$ implies $\ACA_0^+$ over $\RCA_0$.
\end{corollary}

The above corollary coupled with Theorem \ref{thm:car-zda} and Corollary \ref{cor:RTtoFS} implies the equivalence over $\RCA_0$ of $\RT^{!\omega}_2$, $\TS^{!\omega}$ and $\FS^{!\omega}$. However we do not know whether $\RT^{!\omega}_2$ is strongly Weihrauch-reducible to $\TS^{!\omega}$ or $\FS^{!\omega}$.


The remaining question is whether $\RRT^{!\omega}$ codes $\emptyset^{(\omega)}$. In the next section, we shall answer this question negatively in a strong sense: $\RRT^{!\omega}$ does not code any non-computable set.
Some computability-theoretic weak anti-basis results for $\RRT^{!\omega}$ can be obtained by streamlining results from the finite dimensional case. For instance, Csima and Mileti~\cite{CsimaMileti_2009} proved the following theorem:

\begin{theorem}[Csima-Mileti \cite{CsimaMileti_2009}]\label{thm:csi-mil}
For all $n,k\geq 2$ there exists a computable $k$-bounded $f:[\NN]^n \to \NN$ with no infinite $\Sigma^0_n$ rainbow.
\end{theorem}

The proof of Theorem~\ref{thm:csi-mil} is uniform in $n$. Using this uniformity we obtain the following. 

\begin{theorem}
There is a computable instance of $\RRT^{!\omega}_2$ with no arithmetical solution.
\end{theorem}

\begin{proof}
    
For each $n\geq 2$ let $f_n : [\NN]^n \to \NN$  be the $2$-bounded instance of $\RRT^n_2$ with no infinite $\Sigma^0_n$ rainbow given by Theorem \ref{thm:csi-mil}. Let $g:[\NN]^{!\omega} \to \NN$ be defined as follows:
$$g(n, x_0, \dots, x_{n-1}) = \langle n, f_n(x_0, \dots, x_{n-1}) \rangle$$   
The function $g$ is clearly $2$-bounded since each $f_n$ is $2$-bounded. Then for every $g$-rainbow $H$ and every $n$ in $H$, the set $H \setminus [0,n]$ is an $f_n$-rainbow, so $H$ is not $\Sigma^0_n$ by \Cref{thm:csi-mil}.
\end{proof}


Patey~\cite{patey2015somewhere} proved that $\RCA_0 \vdash (\forall n)[\RRT^{n+1}_2 \to \TS^n]$ and for every $n \in \NN^+$, $\RCA_0 \vdash \RRT^{2n+1}_2 \to \FS^n$. The argument almost translates in the exactly $\omega$-largeness setting as follows. A set $s\subseteq\NN$ is \emph{exactly $(\omega+1)$-large} if $s \setminus \{\min s\}$ is exactly $\omega$-large. Let $\RRT^{!(\omega+1)}_2$ be the Rainbow Ramsey Theorem for $2$-bounded colorings of the exactly $(\omega+1)$-large sets.

\begin{proposition}\label[proposition]{ts-omega-sc-rrt-omegaplus1}
$\TS^{!\omega} \leq_{\sW} \RRT^{!(\omega+1)}_2$.
\end{proposition}
\begin{proof}
Fix an instance $f : [\NN]^{!\omega} \to \NN$ of $\TS^{!\omega}$. Then, consider the following $f$-computable instance $g$ of $\RRT^{!(\omega + 1)}_2$: for every $s \in [\NN]^{!\omega}$ and every $x \in \NN$, if $f(s) = \langle x,y \rangle$ with $x < y < \min s$, let $g(x,s) = g(y,s)$ and otherwise assign to $g(x,s)$ a fresh color. The construction of $g$ is uniform in~$f$.

Let $H$ be an infinite rainbow for $g$ and let $x,y \in H$ with $x < y$. The set $H_1 = H \setminus [0,y]$ is $f$-thin for the color $\langle x,y \rangle$. Indeed, for every $s \in [H_1]^{!\omega}$, by definition of $H_1$, $x < y < \min s$, thus, if $f(s) = \langle x, y \rangle$, then $g(x,s)$ would be equal to $g(y,s)$, contradicting the fact that $H$ is a rainbow for $g$. The definition of $H_1$ is uniform in $H$, thus $\TS^{!\omega} \leq_{\sW} \RRT^{!(\omega+1)}_2$.
\end{proof}

\begin{corollary}
There exists a computable instance of $\RRT^{!(\omega+1)}_2$ such that every solution computes $\emptyset^{(\omega)}$.
\end{corollary}
\begin{proof}
Immediate by \Cref{ts-omega-computes-omega-jump} and \Cref{ts-omega-sc-rrt-omegaplus1}.
\end{proof}

\section{Coloring barriers of order type $\omega^\omega$}\label{sec:coloring-barriers}

In the previous sections we established lower bounds showing that the Free Set Theorem and the Thin Set Theorem for the Schreier barrier code $\emptyset^{(\omega)}$. In this section we develop a more robust generalization of the principles of interest, based on the notion of barrier. The main motivation is to prove upper bounds on a generalization of the Large Rainbow Ramsey theorem. 

Ramsey's theorem for exactly $\omega$-large sets ($\RT^{!\omega}_k$) is arguably the simplest generalization of Ramsey's theorem to collections of finite sets of arbitrary size, which is combinatorially true. However, 
the restriction $|s| = 1+\min s$ is somewhat arbitrary, and could be replaced by any restriction of the form $|s| = h(\min s)$ for a computable function $h : \NN \to \NN$, while leaving the computational lower bounds and upper bounds of the statement unchanged. 
In this section, we therefore generalize the previous statements about exactly $\omega$-large sets to a family of statements satisfying better closure properties.

There exist two possible approaches to relate $\RT^{!\omega}_k$ to existing theorems. The bottom-up approach, already explained, consists in considering $\RT^{!\omega}_k$ as a generalization of Ramsey's theorem to larger families of finite sets. The top-down approach, which we explore now, consists in seeing $\RT^{!\omega}_k$ as a particular case of the clopen Ramsey Theorem by Galvin and Prikry~\cite{galvinBorelSetsRamsey1973}. In what follows, we write $[X]^\omega$ for the class of all infinite subsets of~$X$. This notation should not be confused with the set $[X]^{!\omega}$ of all exactly $\omega$-large subsets of~$X$.

\begin{theorem}[Clopen Ramsey Theorem]
Fix $k \in \NN^+$.
For every continuous coloring $f : [\NN]^\omega \to k$, there exists an infinite set~$H \subseteq \NN$
which is \emph{$f$-homogeneous}, that is, such that $|f([H]^\omega)| = 1$.
\end{theorem}

Every coloring $f : [\NN]^{!\omega} \to k$ can be considered as a continuous coloring $g : [\NN]^\omega \to k$
defined by $g(X) = f(X \uh_{1+\min X})$, and every $g$-homogeneous set is $f$-homogeneous.
On the other hand, given a continuous coloring $g : [\NN]^\omega \to k$, there exists a prefix-free set $B \subseteq [\NN]^{<\omega}$ and a coloring $f : B \to k$ such that 
\begin{itemize}
    \item[(1)] for every infinite set~$X \in [\NN]^\omega$, there exists a (unique)~$s_X \in B$ such that $s_X \prec X$ (where $\prec$ denotes proper initial segment);
    \item[(2)] for every $X \in [\NN]^\omega$, $g(X) = f(s_X)$.
\end{itemize}
Any set~$B$ satisfying the above properties is called a \emph{block}.
Marcone~\cite{marcone2005wqo} proved over~$\RCA_0$ that such a set~$B$ can always be assumed to satisfy slightly stronger structural properties, making it something called a \emph{barrier}. Given a set~$B \subseteq [\NN]^{<\omega}$, we write $\base(B)$ for the set $\{ n \in \NN : (\exists s \in B)(n \in s) \}$. In the following definition $\preceq$ denotes the initial segment relation and $\subset$ denotes the proper subset relation.

\begin{definition}
A set $B \subseteq [\NN]^{<\omega}$ is a \emph{barrier} if
\begin{enumerate}
    \item[(1)] $\base(B)$ is infinite;
    \item[(2)] for every~$X \in [\base(B)]^\omega$, there is some~$s \in B$ such that $s \preceq X$;
    \item[(3)] for every~$s, t \in B$, $s \not \subset t$.
\end{enumerate}
\end{definition}

Note that the last item is stronger than asking for $B$ to be prefix-free.
The simplest notions of barriers are the families $[\NN]^n$ for ~$n \in \NN$. 
Barriers were introduced by Nash-Williams~\cite{nash-williams1968better} in order to study \emph{better quasi-orders} (bqo), a strengthening of well-quasi-orders (wqo) with better closure properties. Barriers were studied by Marcone~\cite{marcone2005wqo} in the context of Reverse Mathematics.

\begin{theorem}[Barrier Ramsey Theorem]\label{thm:bart}
Fix a barrier $B \subseteq [\NN]^{<\omega}$ and some~$k \in \NN^+$. 
For every coloring $f : B \to k$, there exists an infinite set $H \subseteq \base(B)$ such that $|f([H]^{<\omega} \cap B)| = 1$. 
\end{theorem}

Marcone~\cite{marcone2005wqo} proved over~$\RCA_0$ that the Barrier Ramsey Theorem is equivalent to~$\ATR_0$. Thus, the Barrier Ramsey Theorem is much stronger than Ramsey's theorem for exactly $\omega$-large sets, which stands at the level of $\ACA_0^+$.

Barriers can be classified based on the order type of their lexicographic order. Given a barrier $B \subseteq [\NN]^{<\omega}$ and $s, t \in B$, let $s <_\lex t$ if $s(x) < t(x)$ for the least~$x$ such that $s(x) \neq t(x)$, if it exists. Here, we identify $s$ and $t$ with finite increasing sequences over~$\NN$. Note that since~$B$ is prefix-free, $<_\lex$ is total on~$B$.
The lexicographic orders are not in general well-orders, but Pouzet~\cite{pouzet1972premeilleurordres} proved that they are on barriers. Assous~\cite{assous1974caracterisation} characterized the order types of lexicographic orders on barriers, and proved that they are either of the form $\omega^n$ for some~$n \in \NN^+$, or $\omega^\alpha \cdot k$ for some~$\alpha \geq \omega$ and $k \in \NN^+$. 

\begin{definition}
The order type of a barrier is the order type of its lexicographic order.
\end{definition}

Clote~\cite{clote1984recursion} proved that the order type of barriers is relevant to the computability-theoretic analysis of the Barrier Ramsey Theorem, by conducting a level-wise analysis of its solutions in the hyperarithmetic hierarchy based on the order type of the barrier. He proved in particular the following theorem:

\begin{theorem}[Clote~\cite{clote1984recursion}]\label{thm:Clote}
Fix~$k \in \NN^+$.
\begin{itemize}
    \item For every computable barrier $B \subseteq [\NN]^{<\omega}$ of order type at most $\omega^\omega$ and every computable coloring $f : B \to k$, there is an infinite $f$-homogeneous set computable in~$\emptyset^{(\omega)}$.
    \item There exists a computable barrier $B \subseteq [\NN]^{<\omega}$ of order type $\omega^\omega$ and a computable coloring $f : B \to 2$ such that every infinite $f$-homogeneous set computes~$\emptyset^{(\omega)}$.
\end{itemize}
\end{theorem}

The Schreier barrier is a simple example of barrier of order type $\omega^\omega$. Carlucci and Zdanowski~\cite{carlucci2014strength} showed that the lower bound of Clote is witnessed by the Schreier barrier.
Based on Clote's analysis, it is natural to conjecture that the Free set, Thin set and Rainbow Ramsey theorems for barriers of order type $\omega^\omega$ are the robust counterpart of their versions for exactly $\omega$-large sets. Actually, we shall see that, arguably, the right notion is the restriction of the statements to a sub-class of barriers of order type $\omega^\omega$.

Given a set~$X$ and $n \in \NN$, we write $[X]^{\leq n}$ for the set of all subsets $s \subseteq X$ such that $|s| \leq n$, and $[X]^{\leq !\omega}$ for the set of all subsets $s \subseteq X$ such that $|s| \leq 1+\min s$. By convention, for~$n = 0$, $[X]^{\leq n}$ is the singleton $\{\emptyset\}$. Given a function $h : \NN \to \NN$, we write $[X]^{\leq h(\cdot)}$ for the set of all finite $s \subseteq X$ such that $|s| \leq h(\min s)$. 

\begin{definition}
A set $B \subseteq [\NN]^{<\omega}$ is \emph{$\omega$-bounded} if and only if $B \subseteq [\NN]^{\leq h(\cdot)}$ for some function~$h : \NN \to \NN$. It is \emph{computably $\omega$-bounded} if furthermore $h$ is computable.
\end{definition}

\begin{lemma}\label[lemma]{barrier-omega-bounded-order type}
A barrier~$B$ has order type at most $\omega^{\omega}$ if and only if $B$ is $\omega$-bounded.
\end{lemma}
\begin{proof}
Suppose first $B \subseteq [\NN]^{\leq h(\cdot)}$ for some function $h : \NN \to \NN$.
For every~$s \in B$, let $\alpha_s = \sum_{i < |s|} \omega^{h(\min s)-i} s(i)$. 
Note that $s <_{\lex} t$ if and only if $\alpha_s < \alpha_t$, so $B$ has order type at most $\omega^\omega$ since $\alpha_s < \omega^{\omega}$ for every $s \in B$.

Suppose now $B$ has order type at most $\omega^{\omega}$. Let $x \in \NN$ and $B_x = \{ s : x \cdot s \in B \}$. Then $B_x$ is a barrier of order type at most $\omega^{n_x}$ for some~$n_x \in \NN^+$. By Assous~\cite[Proposition II.1]{assous1974caracterisation}, a barrier~$B$ has order type at most $\omega^n$ if and only if $B \subseteq [\NN]^{\leq n}$, so $B_x \subseteq [\NN]^{\leq n_x}$. Let $h(x) = n_x$. Then $B \subseteq [\NN]^{\leq h(\cdot)}$.
\end{proof}

A function $h : \NN \to \NN$ is \emph{left-c.e.} if there is a uniformly computable sequence of functions $h_0, h_1, \dots$ such that for every~$x, i \in \NN$, $h_i(x) \leq h_{i+1}(x)$ and $\lim_i h_i(x) = h(x)$. The sequence $(h_i)_{i\in\NN}$ is then called a \emph{left-c.e. approximation} of $h$.
If $B$ is a computable barrier of order type at most $\omega^\omega$, then it is $\omega$-bounded by a left-c.e. function. This bound is tight, as there exist computable barriers of order type $\omega^{\omega}$ which are not computably $\omega$-bounded.


Let us first define a generalized version of Rainbow Ramsey Theorem for subsets of $[\NN]^{<\omega}$, and show that its restriction to computable barriers of order type $\omega^\omega$ codes the jump.

\begin{definition}
Let $B \subseteq [\NN]^{<\omega}$.
A coloring $f : B \to \NN$ is \emph{$k$-bounded} if for every~$c \in \NN$, $|f^{-1}(c)| \leq k$.
A set~$H \subseteq \base(B)$ is an \emph{$f$-rainbow} if for every~$s, t \in B \cap [H]^{<\omega}$ such that $s \neq t$, $f(s) \neq f(t)$.
\end{definition}

In this paper, we shall consider only sets~$B\subseteq [\NN]^{<\omega}$ such that $\base(B) = \NN$.


\begin{definition}[Generalized Rainbow Ramsey Theorem]
Given a set $B \subseteq [\NN]^{<\omega}$ and $k \in \NN$, let $\RRT^B_k$ be the statement \qt{For every $k$-bounded coloring $f : B \to \NN$, there exists an infinite $f$-rainbow}.    
\end{definition}

As seen in \Cref{prop:rrt-full-is-false}, the statement $\RRT^B_k$ is not mathematically true for $B = [\NN]^{<\omega}$ and $k\in \NN$. However, its restriction to barriers follows from Ramsey's theorem for barriers (the proof is completely analogous to the proof of~\Cref{thm:largerrt} above). Moreover, we shall see in \Cref{sec:cone_avoid} that $\RRT^B_k$ holds for every $\omega$-bounded barrier $B \subseteq [\NN]^{<\omega}$.

The following proposition shows that $\RRT^B_2$ restricted to computable barriers of order type~$\omega^\omega$ codes the halting set.

\begin{proposition}\label[proposition]{prop:rrt-ot-omega-omega-code-jump}
    There exist a computable barrier $B$ of order type $\omega^{\omega}$ and a computable 2-bounded function $f : B \to \NN$ such that every $f$-rainbow computes~$\emptyset'$.
\end{proposition}
\begin{proof}
Let $g : \NN \to \NN$ be the modulus of $\emptyset'$ and let $(g_n)_{n \in \NN}$ be a left-c.e. approximation of $g$. It can be assumed that $g_n$ is non-decreasing for each~$n \in \NN$.

Let $B$ be defined as follows: for $x,y \in \NN$ with $x < y$ and $s \subseteq \NN$ with $y < \min s$, let $x \cdot y \cdot s \in B$ if and only if $|s| = (g_{\min s}(x))^2$.

$B$ is a barrier with base $\NN$, indeed, for every infinite set $X = \{x_0,x_1, \dots\}$, $(x_0,x_1, \dots, x_{g_{x_2}(x_0)^2 + 1}) \in B$ and if $x \cdot y \cdot s, x' \cdot y' \cdot s' \in B$ satisfy $x \cdot y \cdot s \subseteq x'\cdot y' \cdot s'$ then $x \geq x'$ and $\min s \geq \min s'$, hence $|s| = (g_{\min s}(x))^2 \geq (g_{\min s'}(x'))^2 = |s'|$, but also $|s| \leq |s'|$, so $|s| = |s'|$ and therefore $x \cdot y \cdot s = x'\cdot y' \cdot s'$. The order type of $B$ is $\omega^{\omega}$ by \Cref{barrier-omega-bounded-order type} as $B$ is $\omega$-bounded by $g^2 + 1$.

Let $h_n : [\NN]^{n+1} \to n$ be the computable instance of $\TS^{n+1}_n$ obtained in \Cref{lem:uniform-computable-thin-coloring} such that every infinite $h_n$-thin set computes $\emptyset'$. Consider also the computable bijection $k : \NN \to \{(y,x) \in \NN^2 : y > x\}$ that list all such pairs lexicographically (the order type of that set is $\omega$).

Let $f : B \to \NN$ be defined as follows: for every $x \cdot y \cdot s \in B$, if $k(h_{|s| - 1}(s)) = (z - x, y - x)$ for some $y < z < \min s$, then let $f(x \cdot y \cdot s) = f(x \cdot z \cdot s)$ and otherwise give a fresh new color for $f(x \cdot y \cdot s)$.

Let $H$ be an infinite rainbow for $f$. For every $x \in H$, there exists some bound $b_x > g(x)$ such that $g_{b_x}(x) = g(x)$. For every $s \subseteq H \setminus [0,b_x]$ of cardinality $g(x)^2$, we have $h_{g(x)^2 - 1}(s) \in [0, g(x)^2 - 2]$. There are two cases:

\textbf{Case 1:} $H \setminus [0,b_x]$ is $h_{g(x)^2 - 1}$-thin for some $x \in H$, in that case, by definition of $h_{g(x)^2 - 1}$, $H \geq_T \emptyset'$.

\textbf{Case 2:} $H \setminus [0,b_x]$ is not $h_{g(x)^2 - 1}$-thin for every $x \in H$, in that case, for every $x \in H$ and every pair $y < z \in (x, x + g(x))$, there exists some $s \subseteq H \setminus [0,b_x]$ such that $k(h_{g(x)^2 - 1}(s)) = (z - x,y - x)$. Indeed, $h_{g(x)^2 - 1}(s)$ takes every value in $[0, g(x)^2 - 2]$, so every such couple $(z - x,y - x)$ is reached and therefore $f(x \cdot y \cdot s) = f(x \cdot z \cdot s)$. Since $H$ is an $f$-rainbow, $y$ and $z$ cannot be both in $H$. So $H \geq \emptyset'$ as for every $x < y < z \in H$, $z$ is bigger than $g(x)$, so $H$ computes a function dominating $g$.

\end{proof}

We shall however see that for every computably $\omega$-bounded barrier $B \subseteq [\NN]^{<\omega}$ and every~$k \in \NN^+$, $\RRT^B_k$ admits strong cone avoidance (see below). Note that the Schreier barrier is an example of a computable, computably $\omega$-bounded barrier.

Because of this, \Cref{prop:rrt-ot-omega-omega-code-jump} cannot be improved to code more than $\emptyset'$. Indeed, every computable barrier of order type $\omega^\omega$ is $\emptyset'$-computably $\omega$-bounded, so for any non-$\emptyset'$-computable set~$D$, every computable barrier of order type $\omega^\omega$, and every computable $k$-bounded function $f : B \to \NN$, there exists an infinite $f$-rainbow which does not compute~$D$.

\section{Large Rainbow Ramsey Theorem avoids cones}\label{sec:cone_avoid}

In this section, we prove that the Rainbow Ramsey theorem for computably $\omega$-bounded barriers admits strong cone avoidance. In particular, this is the case for the Large Rainbow Ramsey theorem since the Schreier barrier is $\omega$-bounded by the computable function $x \mapsto x + 1$.

\begin{definition}
A problem~$\Psf$ admits \emph{strong cone avoidance} if for every set~$Z$, every non-$Z$-computable set~$C$ and every $\Psf$-instance~$X$, there exists a $\Psf$-solution~$Y$ to~$X$ such that $C \not \leq_T Y \oplus Z$.
\end{definition}

Note that in the previous definition, no computability constraint is given on the $\Psf$-instance~$X$. Thus, strong cone avoidance reflects the combinatorial weakness of~$\Psf$, in the sense that no matter how complex the instance is, it cannot code in its solutions an infinite binary sequence. We shall use the following two theorems:

\begin{theorem}[Wang~\cite{wang2014some}]
For every~$n \in \NN$, $\FS^n$, $\TS^n$ and $\RRT^n$ admit strong cone avoidance.
\end{theorem}

Wang~\cite{wang2014some} introduced and studied the following formal theorem, which is strictly related to the Thin Set theorem. 

\begin{definition}[Achromatic Ramsey Theorem]
Let $n, k, \ell \in \NN^+$. For every coloring $f : [\NN]^n \to k$, there exists an infinite $H \subseteq \NN$
such that $|f([H]^n)| \leq \ell$. We denote this statement $\RT^n_{k, \ell}$. We write $\RT^n_{<\infty, \ell}$ for $\forall k \RT^n_{k, \ell}$.
\end{definition}

Note that $\RT^n_{k,k-1}$ is the same as $\TS^n_k$ from \Cref{def:ts-finite}.

The following sequence of numbers, known as \emph{Catalan numbers}, is omnipresent in Combinatorics. It is inductively defined as follows:
$$
C_0 = 1 \hspace{1cm} C_{n+1} = \sum_{i = 0}^n C_i C_{n-i}
$$
This sequence starts with $1, 1, 2, 5, 14, 42, \dots$ (see sequence A000108 in the OEIS). The number $C_n$ admits many characterizations, such as the number of ways of associating $n$ applications of a binary operator. In computability theory, the $n$th Catalan number $C_n$ surprisingly arose as the exact threshold $\ell$ at which $\RT^n_{<\infty, \ell}$ admits strong cone avoidance.

\begin{theorem}[Cholak and Patey~\cite{cholak2020thin}]
For every~$n \in \NN^+$, $\RT^n_{<\infty, C_n}$ admits strong cone avoidance.
\end{theorem}

In the remaining part of this section we show that the Rainbow Ramsey Theorem for computable, computably $\omega$-bounded barriers admits strong cone avoidance and therefore does not code the jump. To obtain this result, we introduce some variants of the Free Set Theorem for large sets which are of interest in their own right. We first introduce the needed terminology.

\begin{definition}
Let $B \subseteq [\NN]^{<\omega}$ be a set and $f : B \to [\NN]^{<\omega}$ be a coloring.
A set~$H \subseteq \base(B)$ is \emph{$f$-free} if for every~$s \in B \cap [H]^{<\omega}$, $f(s) \cap H \subseteq s$.
\end{definition}

For example, given a coloring $f : [\NN]^n \to \NN$, one can let $B = [\NN]^n$ and $g : [\NN]^n \to [\NN]^{<\omega}$ be defined by $g(s) = \{f(s)\}$. Then a set is $f$-free if and only if it is $g$-free.
Of course, even with $B = [\NN]^n$, infinite free sets do not necessarily exist for arbitrary colorings. We need to impose some constraints on the size of the sets in the image of~$f$.

\begin{definition}
Fix $B \subseteq [\NN]^{<\omega}$.
A coloring~$f : B \to [\NN]^{<\omega}$ is \emph{$b$-constrained} for a bounding function $b : \NN \to \NN$ if for every~$s \in B$, $|f(s)| \leq b(\min s)$. If $b$ is the constant function $x \mapsto k$, then we say that $f$ is \emph{$k$-constrained}.
\end{definition}

\begin{definition}[$k$-Constrained Free Set Theorem]
Fix $B \subseteq [\NN]^{<\omega}$ and $k \in \NN$.
$\FS^B_k$ is the statement \qt{For every coloring  $f : B \to [\NN]^{\leq k}$, there is an infinite $f$-free set}.
\end{definition}

Given $n \in \NN$ and a function $h : \NN \to \NN$, we write $\FS^{\leq n}_k$, $\FS^{\leq !\omega}_k$ and $\FS^{\leq h(\cdot)}_k$ for $\FS^B_k$ when $B$ is $[\NN]^{\leq n}$, $[\NN]^{\leq !\omega}$ and $[\NN]^{\leq h(\cdot)}$, respectively.
In the extreme case where~$B = \{\emptyset\}$, $\FS^B_k$ is nothing but the statement \qt{For every finite set $F \in [\NN]^{\leq k}$, there is an infinite set~$H \subseteq \NN$ such that $H \cap F = \emptyset$}. We have seen in \Cref{fs-omega-computes-omega-jump} that there exists a computable instance~$f$ of $\FS^{\leq !\omega}_1$ such that every infinite $f$-free set computes~$\emptyset^{(\omega)}$. The case $n < \omega$ is different.

\begin{proposition}\label[proposition]{prop:fs-n-k-strong-cone-avoidance}
For every $k, n \in \NN$, $\FS^{\leq n}_k$ admits strong cone avoidance.
\end{proposition}
\begin{proof}
Fix some set~$Z$, some non-$Z$-computable set~$C$, and some coloring $f : [\NN]^{\leq n} \to [\NN]^{\leq k}$.
For every~$m \leq n$ and $j < k$, let $f_{m, j} : [\NN]^m \to \NN$ be the coloring defined for every~$s \in [\NN]^m$ by letting $f_{m, j}(s)$ be the $j$th element of~$f(s)$, if it exists, and $f_{m, j}(s) = 0$ otherwise. By finitely many successive applications of strong cone avoidance of $\FS^m$ for $m \leq n$ (see Wang~\cite{wang2014some}), there is an infinite set~$H \subseteq \NN$ which is simultaneously $f_{m,j}$-free for every $m \leq n$ and $j < k$, and such that $C \not \leq_T H \oplus Z$.

We claim that $H$ is $f$-free. Suppose by way of contradiction that there is some~$s \in [H]^{\leq n}$ and some~$c \in (f(s) \cap H) \setminus s$. 
Let $j$ be such that $c$ is the $j$th element of~$f(s)$, and let $m = |s|$. Then $f_{m,j}(s) = c$, contradicting $f_{m,j}$-freeness of~$H$.
\end{proof}

Note that the $k$-constraint cannot be released, even in the case of colorings of singletons, as it would yield a combinatorially false statement:

\begin{proposition}
There exists a computable function $f : [\NN]^1 \to [\NN]^{<\omega}$ such that for every~$x \in \NN$, $|f(\{x\})| \leq x$, and with no $f$-free set of size~2.
\end{proposition}
\begin{proof}
Let $f(\{x\}) = [0, x)$. Let $\{x, y\}$ be an $f$-free set, with $x < y$. Then $x \in f(\{y\}) \setminus \{y\}$, contradicting $f$-freeness of~$\{x,y\}$.
\end{proof}

One can however replace the constant constraint by a function when considering a natural sub-class of instances.

\begin{definition}
Fix $B \subseteq [\NN]^{<\omega}$.
A function $f : B \to [\NN]^{<\omega}$ is \emph{progressive} if for every~$s \in B$, either $f(s) = \emptyset$, or $\min f(s)\geq \min s$.
\end{definition}

An easy combinatorial argument shows that the following statement is classically true.

\begin{definition}[Progressive Free Set Theorem]
Fix a set $B \subseteq [\NN]^{<\omega}$ and a bounding function $b : \NN \to \NN$.
$\PFS^B_b$ is the statement \qt{For every $b$-constrained progressive coloring $f : B \to [\NN]^{<\omega}$, there is an infinite $f$-free set}.
\end{definition}

Here again, given $n \in \NN$ and a function $h : \NN \to \NN$, we write $\PFS^{\leq n}_b$, $\PFS^{\leq !\omega}_b$ and $\PFS^{\leq h(\cdot)}_b$ for $\PFS^B_b$ when $B$ is $[\NN]^{\leq n}$, $[\NN]^{\leq !\omega}$ and $[\NN]^{\leq h(\cdot)}$, respectively.
We now show how the above principle relates to the Generalized Rainbow Ramsey Theorem $\RRT^B$.

\begin{proposition}\label[proposition]{rrt-omega-progressive-fs-omega}
Fix a barrier $B \subseteq [\NN]^{<\omega}$ and some~$k \in \NN$.
For every $k$-bounded coloring $f : B \to \NN$, there is an $f'$-computable $k$-constrained progressive coloring $g : B \to [\NN]^{<\omega}$ such that every infinite $g$-free set is an $f$-rainbow.
\end{proposition}
\begin{proof}
Let $\leq_{\lex}$ be the lexicographic ordering on $B$ (seen as finite increasing sequences over $\NN^{<\omega}$). 
In particular, if $s <_{\lex} t$, then $\min s \leq \min t$, so, as $t \not \subseteq s$ by definition of a barrier, $t \setminus s$ is not empty and $\min(t \setminus s) \geq \min s$.

Let $g(s) = \{\min(t \setminus s) : f(t) = f(s) \wedge s <_{\lex} t\}$. The coloring $g$ is progressive and $k$-constrained. Note that $g$ is computable in the Turing jump of~$f$, as one needs to make an unbounded search for all $t >_{\lex} s$.

We claim that every infinite $g$-free set~$H$ is an $f$-rainbow. Suppose by way of contradiction that there are some distinct~$s, t \in [H]^{<\omega} \cap B$ such that $f(s) = f(t)$. One can suppose without loss of generality that $s <_{\lex} t$. Then $\min(t \setminus s) \in (g(s) \cap H) \setminus s$, contradicting the $g$-freeness of~$H$.
\end{proof}



The following proposition shows the existence of a computable barrier of order type $\omega^\omega$ for which the Progressive Free Set theorem does not admit cone avoidance. In particular, this barrier is not computably $\omega$-bounded, as we shall prove that the Progressive Free Set theorem for computable barriers which are computably $\omega$-bounded admits strong cone avoidance.

\begin{proposition}\label[proposition]{prop:computable-barrier-pfs-computes-zp}
    There exists a computable barrier $B$ of order type $\omega^{\omega}$ and a computable progressive coloring $f : B \to \NN$ such that every $f$-free set computes~$\emptyset'$.
\end{proposition}

\begin{proof}
Let $g : \NN \to \NN$ be the modulus of $\emptyset'$ and let $(g_n)_{n \in \NN}$ be a left-c.e. approximation of $g$. Without loss of generality, it can be assumed that $g_{n}(x) \geq 1$ for every $x,n \in \NN$ and that $g_n$ is non-decreasing for each~$n \in \NN$.

Let $B$ be defined as follows: for $x \in \NN$ and $s \subseteq \NN$ with $x < \min s$, let $x \cdot s \in B$ if and only if $|s| = g_{\min s}(x)$. 


We claim that $B$ is a barrier with base $\NN$: for any infinite subset $X = \{x_0, x_1, \dots\}$ of $\NN$, we have $\{x_0, x_1, \dots, x_{g_{x_1}(x_0)}\}\in B$ (since $g_{x_1}(x_0)\geq 1$). Let $x \cdot s, y \cdot t$ be in $B$ and suppose $x\cdot s \subseteq y \cdot t$.
Then $x \geq y$ and $\min s \geq \min t$. Moreover $|s| = g_{\min s}(x) \geq g_{\min s}(y)\geq g_{\min t}(y) = |t|$. Thus $s = t$ and $x=y$. By \Cref{barrier-omega-bounded-order type}, the order type of~$B$ is $\omega^{\omega}$.

Let $f : B \to \NN$ be defined by $f(x \cdot s) = x + h_{|s| - 1}(s) + 1$ where $h_n : [\NN]^{n+1} \to n$ is the computable instance of $\TS^{n+1}_n$ obtained in \Cref{lem:uniform-computable-thin-coloring} such that every infinite $h_n$-thin set computes $\emptyset'$. 
Let $H$ be an infinite $f$-free set. For every $x \in H$, there exists some bound $b_x > g(x)$ such that $g_{b_x}(x) = g(x)$. For every $s \subseteq H \setminus [0,b_x]$ of cardinality $g(x)$, we have $f(x \cdot s) = x + h_{g(x) - 1}(s) + 1$. There are two cases:
\smallskip

\textbf{Case 1:} $H \setminus [0,b_x]$ is $h_{g(x) - 1}$-thin for some $x \in H$, in that case, by definition of $h_{g(x) - 1}$, $H \geq_T \emptyset'$.
\smallskip

\textbf{Case 2:} $H \setminus [0,b_x]$ is not $h_{g(x)}$-thin for every $x \in H$, in that case, for every $x \in H$ and every $y \in (x,x + g(x))$, there exists some $s \subseteq H \setminus [0,b_x]$ such that $f(x \cdot s) = y$ and therefore, since $H$ is $f$-free, $y$ is not in $H$. So $H \geq_T \emptyset'$ as the principal function of $H$ dominates $g$.
\end{proof}

We now proceed to establish that the $\PFS^{\leq h(\cdot)}_h$ admits strong cone avoidance for every computable function $h : \NN \to \NN$. Note that $h$ needs to be computable since by \Cref{prop:computable-barrier-pfs-computes-zp}, letting $h$ be the function such that $B \subseteq [\NN]^{\leq h(\cdot)}$, $\PFS^{\leq h(\cdot)}_1$ does not admit cone avoidance. 

\begin{theorem}\label[theorem]{thm:pfs-omega-strong-cone-avoidance}
Fix a set~$Z$, a non-$Z$-computable set~$D$ and a $Z$-computable function $h : \NN \to \NN$.
For every $h$-constrained progressive coloring $f : [\NN]^{\leq h(\cdot)} \to [\NN]^{<\omega}$, there exists an infinite $f$-free set $G \subseteq \NN$ such that $D \not \leq_T G \oplus Z$.
\end{theorem}

\begin{proof}
For simplicity, we prove the theorem in a non-relativized form. Relativization is straightforward.
Fix a non-computable set~$D$. 
Let $\P$ be the collection of all progressive colorings of type $[\NN]^{\leq h(\cdot)} \to [\NN]^{<\omega}$. For a function $b:\NN\to\NN$ we write $\P_b$ for the class of all $b$-constrained colorings in~$\P$.

Given two colorings $f, g \in \P$, we write $g \leq f$ if for every~$s \in [\NN]^{\leq h(\cdot)}$, $g(s) \supseteq f(s)$.
Note that if $H$ is $g$-free and $g \leq f$, then $H$ is $f$-free.
Given two colorings $f, g \in \P$, let $f \cup g$ be the coloring defined by $(f \cup g)(s) = f(s) \cup g(s)$. The coloring $f \cup g$ is the greatest lower bound of~$f$ and $g$ with respect to $\leq$.

Consider the following notion of forcing:

\begin{definition}
A \emph{condition} is a triple $(f, \sigma, X)$ such that
$f \in \P$ is $b_f$-constrained, for some computable function $b_f : \NN \to \NN$,
$\sigma \in [\NN]^{<\omega}$, $X \subseteq \NN$ is an infinite set such that $\max \sigma < \min X$, and
\begin{enumerate}
    \item[(a)] for every~$s \in [\sigma \cup X]^{\leq h(\cdot)}$ with $\min s \in \sigma$, $f(s) \cap X \subseteq s$.
    \item[(b)] for every~$s \in [\sigma \cup X]^{\leq h(\cdot)}$, $f(s) \cap \sigma \subseteq s$.
    \item[(c)] $D \not \leq_T X$.
\end{enumerate}
A condition $d = (g, \tau, Y)$ \emph{extends} $c = (f, \sigma, X)$ (written $d \leq c$) if
$g \leq f$, $\tau \succeq \sigma$, $Y \subseteq X$ and $\tau \setminus \sigma \subseteq X$.
\end{definition}

The following lemma states that property (a) can be obtained \qt{for free}, that is, by restricting the reservoir, and therefore does not impose any constraint on the stem.
In what follows, fix a coloring $f \in \P$ which is $b_f$-constrained, for some computable function $b_f : \NN \to \NN$.

\begin{lemma}\label[lemma]{lem:fs-preserving-a}
For every $\sigma \in [\NN]^{<\omega}$ and every infinite set~$X$ such that $\max \sigma < \min X$ and $D \not \leq_T X$,
there is an infinite set $Y \subseteq X$ such that $D \not \leq_T Y$ and 
for every~$s \in [\sigma \cup Y]^{\leq h(\cdot)}$ with $\min s \in \sigma$, $f(s) \cap Y \subseteq s$.
\end{lemma}
\begin{proof}
For every~$t \in [\sigma]^{\leq h(\cdot)}$ with $t \neq \emptyset$, let $f_t : [X]^{\leq h(\min t) - |t|} \to [\NN]^{\leq b_f(\min t)}$ be defined for every~$u$ by $f_t(u) = f(t \cdot u)$.
By finitely many successive applications of \Cref{prop:fs-n-k-strong-cone-avoidance},
there exists an infinite subset~$Y \subseteq X$ such that $D \not \leq_T Y$
and such that $Y$ is simultaneously $f_t$-free for every~$t \in [\sigma]^{\leq h(\cdot)}$ with $t \neq \emptyset$.

We claim that $Y$ is our desired set. Indeed, for every~$s \in [\sigma \cup Y]^{\leq h(\cdot)}$ with $\min s \in \sigma$, letting $t = s \cap \sigma$ and $u = s \cap Y$, we have $f(s) = f_t(u)$, so by $f_t$-freeness of~$Y$, $f(s) = f_t(u) \cap Y \subseteq u \subseteq s$.
\end{proof}

In what follows, recall that $C_n$ stands for the $n$th Catalan number.

\begin{definition}
A set $X$ \emph{stabilizes} $\sigma$ if for every $t \in [\sigma]^{\leq h(\cdot)}$ with $t \neq \emptyset$ and every~$n \leq h(\min t) - |t|$,
there is a set $I_{t,n} \subseteq [\sigma]^{\leq b_f(\min t) \times C_n}$ such that for every~$u \in [X]^n$,
$f(t, u) \cap \sigma \subseteq I_{t, n}$.
\end{definition}

\begin{lemma}\label[lemma]{lem:fs-stabilizing}
For every $\sigma \in [\NN]^{<\omega}$ and every infinite set~$X$ such that $\max \sigma < \min X$ and $D \not \leq_T X$,
there is an infinite set $Y \subseteq X$ stabilizing $\sigma$ such that $D \not \leq_T Y$.
\end{lemma}
\begin{proof}
For every~$t \in [\sigma]^{\leq h(\cdot)}$ with $t \neq \emptyset$ and every~$n \leq h(\min t) - |t|$, let $g_{t, n} : [X]^n \to [\sigma]^{\leq b_f(\min t)}$ be defined by $g_{t, n}(u) = f(t \cdot u) \cap \sigma$. By finitely many successive applications of strong cone avoidance of $\RT^n_{<\infty, C_n}$ (see Cholak and Patey~\cite{cholak2020thin}), there is an infinite subset $Y \subseteq X$ such that $D \not \leq_T Y$ and for every~$t \in [\sigma]^{\leq h(\cdot)}$ with $t \neq \emptyset$ and every~$n \leq h(\min t) - |t|$, $|g_{t, n}([Y]^n)| \leq C_n$. Note that here, a color is an element of $[\sigma]^{\leq b_f(\min t)}$ instead of a natural number. Let $I_{t, n} = \bigcup g_{t, n}([Y]^n)$. Then $|I_{t, n}| \leq b_f(\min t) \times C_n$. By definition, for every~$u \in [Y]^n$, $f(t \cdot u) \cap \sigma = g_{t, n}(u) \subseteq I_{t, n}$.
\end{proof}

Given a computable function $b : \NN \to \NN$, let $b ^{+} : \NN \to \NN$ be the computable function defined by  $b^{+}(m) = \sum_{n \leq h(m)} b(m) \times C_n$.

\begin{definition}
Let $X$ be a reservoir stabilizing $[0, k]$. The \emph{limit coloring} is the function $g_{k,X} : [k]^{\leq h(\cdot)} \to [k]^{<\omega}$ defined by $g_{k,X}(\emptyset) = \emptyset$ and $g_{k,X}(t) = \bigcup_{n \leq h(t) - |t|} I_{t, n}$ otherwise.
\end{definition}

The limit function $g_{k,X}$ is $b_f^{+}$-constrained. Note that if $\rho \subseteq [0, k]$ is $g_{k,X}$-free, then it is $f$-free. Indeed, for every~$t \in [0, k]^{\leq h(\cdot)}$, $g_{k,X}(t) = I_{t, 0}$ and by definition of stability for~$n = 0$, $f(t) \cap [0, k] \subseteq I_{t, 0}$.

The following lemma is the core combinatorial lemma which specifies the conditions under which a block of elements~$\rho$ can be added to the stem while preserving the property~(b).

\begin{lemma}\label[lemma]{lem:fs-preserving-b}
Let $(f, \sigma, X)$ be a condition
and $Y \subseteq X$ be an infinite set stabilizing $[0, k]$ for some~$k \in \NN$ and let $g_{k,Y} : [k]^{\leq h(\cdot)} \to [k]^{<\omega}$ be the limit coloring.
Let $\rho \subseteq X \uh_k$ be a finite $g_{k,Y}$-free set. Then $(f, \sigma \cup \rho, Y)$
satisfies property~(b).
\end{lemma}
\begin{proof}
Fix some~$s \in [\sigma \cup \rho \cup Y]^{\leq h(\cdot)}$. We have multiple cases.
\begin{itemize}
    \item Case 1: $s \cap \sigma \neq \emptyset$. Then by properties (a) and (b) of $(f, \sigma, X)$, $f(s) \cap (\sigma \cup X) \subseteq s$. Since $(\sigma \cup \rho) \subseteq (\sigma \cup X)$, then $f(s) \cap (\sigma \cup \rho) \subseteq s$.
    \item Case 2: $s \cap \sigma = \emptyset$ but $s \cap \rho \neq \emptyset$. By property (b) of $(f, \sigma, X)$, $f(s) \cap \sigma \subseteq s$. Let $t = s \cap \rho$ and $u = s \cap Y$. Since $\rho$ is $g_{k,Y}$-free, $g_{k,Y}(t) \cap \rho \subseteq t$. By definition of $g_{k,Y}(t)$, $I_{t, |u|} \subseteq g_{k,Y}(t)$, so $I_{t, |u|} \cap \rho \subseteq t$. In particular, $f(s) = f(t, u) \in I_{t, |u|}$, so $f(s) \cap \rho \subseteq t \subseteq s$. Thus, $f(s) \cap (\sigma \cup \rho) \subseteq s$.
    \item Case 3: $s \cap (\sigma \cup \rho) = \emptyset$. Then $s \subseteq Y$. Since $\min Y > \max (\sigma \cup \rho)$, then by progressiveness of~$f$, $f(s) \cap (\sigma \cup \rho) = \emptyset$.
\end{itemize}
\end{proof}

One can combine \Cref{lem:fs-preserving-a,lem:fs-stabilizing,lem:fs-preserving-b} to obtain an extensibility lemma, saying that every sufficiently generic filter induces an infinite set. 

\begin{lemma}\label[lemma]{lem:fs-extendibility}
Let $(f, \sigma, X)$ be a condition. There is an extension $(f, \tau, Y) \leq (f, \sigma, X)$ such that $|\tau| > |\sigma|$.
\end{lemma}
\begin{proof}
Let $x = \min X$. By \Cref{lem:fs-stabilizing}, there is an infinite subset $Y_0 \subseteq X$ stabilizing $[0, x]$ such that $D \not \leq_T Y_0$. Let $g$ be the limit function. By \Cref{lem:fs-preserving-a}, there is an infinite subset $Y \subseteq Y_0$ such that $(f, \sigma \cup \{x\}, Y)$ satisfies property (a).
By \Cref{lem:fs-preserving-b}, $\{x\}$ being vacuously $g$-free, $(f, \sigma \cup \{x\}, Y)$ satisfies property (b). Thus, $(f, \sigma \cup \{x\}, Y)$ is a valid extension.
\end{proof}

Given a computable function $b : \NN \to \NN$, the space $\P_b$ is not compact.
However, $\P_b$ is in one-to-one correspondence with the effectively compact space $\R_b$ of all relations $R \subseteq [\NN]^{\leq h(\cdot)} \times \NN$ such that for every~$s \in [\NN]^{\leq h(\cdot)}$, $|\{ y : (s, y) \in R \}| \leq b(\min s)$ and for every~$(s, y) \in R$, $y \geq \min s$. Indeed, given a function $g \in \P_b$, one can define the relation $R_g \in \R_b$ defined by $R_g = \{ (s, y) : y \in g(s) \}$ and given a relation $R \in \R_b$, the function $g_R \in \P_b$ is defined by $g_R(s) = \{ y : (s, y) \in R \}$.

Note that the map $g \mapsto R_g$ is computable, but the map $R \mapsto g_R$ is not even continuous. Thankfully, there is a Turing functional which, given $R$ and a finite set~$\rho$, decides whether $\rho$ is $g_R$-free or not. Indeed, to decide whether $\rho$ is $g_R$-free, one does not need to know $g_R$ restricted to $[\rho]^{\leq h(\cdot)}$, only to know $R$ restricted to $[\rho]^{\leq h(\cdot)} \times \rho$.

We are now ready to define the forcing question:

\begin{definition}
Let $(f, \sigma, X)$ be a condition and $\varphi(G)$ be a $\Sigma^0_1$-formula.
Let $(f, \sigma, X) \qvdash \varphi(G)$ hold if and only if for every relation $R \in \R_{b_f^{+}}$, there is a finite $g_R$-free set $\rho \subseteq X$ such that $\varphi(\sigma \cup \rho)$ holds.
\end{definition}

The previous formulation of the forcing question is $\Pi^1_1(X)$ as it starts with a second-order universal quantification. However, thanks to the effective compactness of the space $\R_{b_f^{+}}$, it is equivalent to a $\Sigma^0_1(X)$-formula:

\begin{lemma}\label[lemma]{lem:fs-forcing-question-sigma01}
Let $(f, \sigma, X)$ be a condition and $\varphi(G)$ be a $\Sigma^0_1$-formula.
Then $(f, \sigma, X) \qvdash \varphi(G)$ if and only if there is some~$k \in \omega$ such that for every $b_f^{+}$-constrained progressive function $g : [0, k]^{\leq h(\cdot)} \to [0, k]^{<\omega}$, there is a finite $g$-free set $\rho \subseteq X \uh_k$ such that $\varphi(\sigma \cup \rho)$ holds.
\end{lemma}
\begin{proof}
Suppose first that there is some~$k \in \NN$ such that for every $b_f^{+}$-constrained progressive function $g : [0, k]^{\leq h(\cdot)} \to [0, k]^{<\omega}$, there is a finite $g$-free set $\rho \subseteq X \uh_k$ such that $\varphi(\sigma \cup \rho)$ holds. Let $R \in \R_{b_f^+}$ be a relation,
and let $g_R \in \P_{b_f^+}$ be the corresponding function. Then, letting $g : [0, k]^{\leq h(\cdot)} \to [0, k]^{<\omega}$ be defined by $g(s) = g_R(s) \cap [0, k]$, the function $g$ is $b_f^{+}$-constrained and progressive, so there is some finite $g$-free set $\rho \subseteq X \uh_k$ such that $\varphi(\sigma \cup \rho)$ holds. In particular, $\rho$ is $g_R$-free. Since there is such a $\rho$ for every $R \in \R_{b_f^+}$, then $(f, \sigma, X) \qvdash \varphi(G)$ holds.

Suppose now that for every~$k \in \NN$, there is a $b_f^{+}$-constrained progressive function $g_k : [0, k]^{\leq h(\cdot)} \to [0, k]^{<\omega}$ such that for every $g_k$-free set $\rho \subseteq X \uh_k$, $\varphi(\sigma \cup \rho)$ does not hold. Let $\T$ be the tree which, at level~$k$, contains all such functions $g_k$, and which is ordered by the function extension relation. The tree $\T$ is finitely branching, so by K\"onig's lemma, there is an infinite path $g \in \T$. This path is a function $g \in \P_{b_f^+}$ such that for every finite $g$-free set $\rho \subseteq X$, $\varphi(\sigma \cup \rho)$ does not hold. Then the relation $R_g$ witnesses that $(f, \sigma, X) \qvdash \varphi(G)$ does not hold.
\end{proof}

Any filter $\F$ for this notion of forcing induces a set 
$$
G_\F = \bigcup \{ \sigma : (g, \sigma, X) \in \F \}
$$
We say that a condition $p$ \emph{forces} a formula $\varphi(G)$ if $\varphi(G_\F)$ holds for every sufficiently generic filter~$\F$ containing~$p$.
The following lemma states that the forcing question meets its specification.

\begin{lemma}\label[lemma]{lem:fs-forcing-question-spec}
Let $p = (f, \sigma, X)$ be a condition and $\varphi(G)$ be a $\Sigma^0_1$-formula.
\begin{enumerate}
    \item[1.] If $p \qvdash \varphi(G)$, then there is an extension of~$p$ forcing $\varphi(G)$.
    \item[2.] If $p \nqvdash \varphi(G)$, then there is an extension of~$p$ forcing $\neg \varphi(G)$.
\end{enumerate}
\end{lemma}
\begin{proof}
Suppose first $p \qvdash \varphi(G)$ holds. By \Cref{lem:fs-forcing-question-sigma01}, there is some~$k \in \NN$ such that for every $b_f^{+}$-constrained progressive function $g : [0, k]^{\leq h(\cdot)} \to [0, k]^{<\omega}$, there is a finite $g$-free set $\rho \subseteq X$ such that $\varphi(\sigma \cup \rho)$ holds. By \Cref{lem:fs-stabilizing}, there is an infinite subset $Y_0 \subseteq X$ stabilizing $[0, k]$ and such that $D \not \leq_T Y_0$. Let $g_{k,Y_0} : [0, k]^{\leq h(\cdot)} \to [0, k]^{<\omega}$ be the limit function. Note that $g_{k,Y_0}$ is $b_f^{+}$-constrained and progressive, so there is a finite $g_{k,Y_0}$-free set $\rho \subseteq X$ such that $\varphi(\sigma \cup \rho)$ holds.
By \Cref{lem:fs-preserving-b}, $(f, \sigma \cup \rho, Y_0)$ satisfies (b).
By \Cref{lem:fs-preserving-a}, there is an infinite subset $Y \subseteq Y_0$ such that $(f, \sigma \cup \rho, Y)$ satisfies (a) and $D \not \leq_T Y$. Thus $(f, \sigma \cup \rho, Y)$ is a valid extension of~$p$. By choice of~$\rho$, it forces $\varphi(G)$.

Suppose $p \nqvdash \varphi(G)$ holds. Then there is a relation $R \in \R_{b_f^{+}}$ such that for every finite $g_R$-free set $\rho \subseteq X$, $\varphi(\sigma \cup \rho)$ does not hold. Let $g_R \in \P_{b_f^{+}}$ be the corresponding function. Let $\hat{f} : [\NN]^{\leq h(\cdot)} \to [\NN]^{<\omega}$ be defined by $\hat{f}(s) = f(s) \cup (g_R(s) \setminus \sigma)$. Note that $\hat{f} \leq f$, and every $\hat{f}$-free subset $\rho \subseteq X$ is $g_R$-free. Moreover, $\hat{f}$ is $(b_f + b_f^{+})$-constrained and progressive. Since $(f, \sigma, X)$ satisfies (b) and for every $s \in [\sigma \cup X]^{\leq h(\cdot)}$, $\hat{f}(s)\cap \sigma\subseteq f(s)\cap \sigma\subseteq s$, then $(\hat{f}, \sigma, X)$ satisfies (b). By \Cref{lem:fs-preserving-a}, there is an infinite subset $Y \subseteq X$ such that $D \not \leq_T Y$ and $(\hat{f}, \sigma, Y)$ satisfies (a). Thus $q = (\hat{f}, \sigma, Y)$ is a valid extension of~$p$. 

We claim that $q$ forces $\neg \varphi(G)$. Indeed, suppose there is an extension $(\tilde{f}, \tau, Z) \leq q$ such that $\varphi(\tau)$ holds. Then by property (b) of  $(\tilde{f}, \tau, Z)$, $\tau$ is $\tilde{f}$-free, and since $\tilde{f} \leq \hat{f}$, $\tau$ is $\hat{f}$-free. Let $\rho = \tau \setminus \sigma$. By definition of~$\hat{f}$, $\rho$ is a $g_R$-free subset of~$X$, contradicting our choice of~$g_R$.
\end{proof}

We can now prove our diagonalization lemma.

\begin{lemma}\label[lemma]{lem:fs-diagonalization}
Let $p = (f, \sigma, X)$ be a condition and $\Phi_e$ be a Turing functional.
There is an extension of~$p$ forcing $\Phi_e^G \neq D$.
\end{lemma}
\begin{proof}
Let $U = \{ (x, v) \in \NN \times 2 : p \qvdash \Phi^G_e(x) \downarrow = v \}$. We have three cases:
\begin{itemize}
    \item Case 1: $(x, 1-D(x)) \in U$ for some~$x \in \NN$. By \Cref{lem:fs-forcing-question-spec}, there is an extension of~$p$ forcing $\Phi^G_e(x)\downarrow = 1-D(x)$, hence forcing $\Phi_e^G \neq D$.
    \item Case 2: $(x, D(x)) \not \in U$ for some~$x \in \NN$. By \Cref{lem:fs-forcing-question-spec}, there is an extension of~$p$ forcing $\neg (\Phi^G_e(x)\downarrow = D(x))$, hence forcing $\Phi_e^G \neq D$.
    \item Case 3: $U$ is the graph of the characteristic function of~$D$. By \Cref{lem:fs-forcing-question-sigma01}, the set $U$ is $\Sigma^0_1(X)$, so $D \leq_T X$, contradiction.
\end{itemize}
\end{proof}

We are now ready to prove \Cref{thm:pfs-omega-strong-cone-avoidance}.
Let $f : [\NN]^{\leq h(\cdot)} \to [\NN]^{<\omega}$ be an $h$-constrained, progressive coloring, for a computable function $h : \NN \to \NN$.
Let $\F$ be a sufficiently generic filter containing $(f, \emptyset, \NN)$,
and let $G_\F = \bigcup \{ \sigma : (g, \sigma, X) \in \F \}$. By definition of a forcing condition, $G_\F$ is $f$-free. By \Cref{lem:fs-extendibility}, $G_\F$ is infinite, and by \Cref{lem:fs-diagonalization}, $D \not \leq_T G_\F$.
This completes the proof of \Cref{thm:pfs-omega-strong-cone-avoidance}.
\end{proof}

Note that the previous theorem is tight in many senses. First, as mentioned previously, by \Cref{prop:computable-barrier-pfs-computes-zp}, there exists some non-computable function $h$ such that $\PFS^{\leq h(\cdot)}_1$ does not even admit cone avoidance. The following proposition shows that if $b$ is allowed to be non-computable, then $\PFS^{\leq 1}_b$ does not admit strong cone avoidance in general.

\begin{proposition}
There exists a $\emptyset'$-computable function $b : \NN \to \NN$ and a $\emptyset'$-computable $b$-constrained progressive function $f : \NN \to [\NN]^{<\omega}$ such that every infinite $f$-free set computes~$\emptyset'$.
\end{proposition}
\begin{proof}
Let $b : \NN \to \NN$ be the modulus of $\emptyset'$,
and let $f(x) = [x+1, \dots, b(x)]$. Let $H$ be an infinite $f$-free set.
Then given $x < y \in H$, $y > b(x)$, so one can $H$-compute a function dominating~$b$, hence $H$-compute $\emptyset'$.
\end{proof}

\begin{remark}
The proof of \Cref{thm:pfs-omega-strong-cone-avoidance} can be adapted to prove many other notions of avoidance or preservation. For instance, one can prove that for every computable function $h : \NN \to \NN$, $\PFS^{\leq h(\cdot)}_h$ admits strong PA avoidance (see Liu~\cite{liu2012rt22}), strong constant-bound enumeration avoidance (see Liu~\cite{liu2015cone}) or strong preservation for $k$ hyperimmunities for every~$k \in \NN$ (see Patey~\cite{patey2017iterative}).
In particular, for every computable function $h : \NN \to \NN$ and $k \in \NN$, $\PFS^{\leq h(\cdot)}_h$ does not imply any of $\WKL_0$, $\WWKL_0$, $\RT^2_2$, $\RT^2_{<\infty, k}$ over $\omega$-models.
\end{remark}



\bigskip

\begin{corollary}
For every computably $\omega$-bounded barrier $B \subseteq [\NN]^{<\omega}$ and every~$k \in \NN^+$, $\RRT^B_k$ admits strong cone avoidance. 
\end{corollary}
\begin{proof}
Fix a set $Z$ and let $D$ be a non-$Z$-computable set.

Fix $k \in \NN^+$ and let $B \subseteq [\NN]^{<\omega}$ be a computably $\omega$-bounded barrier, with $h : \NN \to \NN$ the computable bound such that $B \subseteq [\NN]^{\leq h(\cdot)}$. Without loss of generality, we can assume that $h(x) \geq k$ for every $x \in \NN$. Let $f : B \to \NN$ be an instance of $\RRT^{B}_k$ and let $g : B \to [\NN]^{<\omega}$ be the $f'$-computable $k$-constrained (and hence $h$-constrained by our assumption that $g(x) \geq k$ for every $x \in \NN$) progressive coloring obtained in \Cref{rrt-omega-progressive-fs-omega}. 

$g$ can be extended to an $h$-constrained progressive coloring $\tilde{g} : [\NN]^{\leq h(\cdot)} \to [\NN]^{<\omega}$ be letting $\tilde{g}(s) = \emptyset$ for every $s \in  [\NN]^{\leq h(\cdot)} \setminus B$.

Then, by \Cref{thm:pfs-omega-strong-cone-avoidance}, there exists an infinite $\tilde{g}$-free set $G \subseteq \NN$ such that $D \not \leq_T G \oplus Z$. By construction of $g$, $G$ is also an infinite $f$-rainbow.
\end{proof}

\bigskip

\section{Conclusions and perspectives}\label{sec:conclusion}

The analysis of the exactly $\omega$-large counterparts to the Ramsey, Free Set and Thin Set theorems from a computable perspective, gave the exact same tight bound, namely, $\emptyset^{(\omega)}$, translating in reverse mathematical terms by an equivalence with~$\ACA_0^+$. This equivalence is to be put in contrast with the finite-dimensional cases, where $\RT^n_2$ coincides with $\ACA_0$ for $n \geq 3$, while $\FS^n$ and $\TS^n$ both admit strong cone avoidance.
On the other hand, the Rainbow Ramsey theorem for exactly $\omega$-large sets ($\RRT^{!\omega}_k$) still has no coding power, and admits strong cone avoidance. \Cref{fig:implicationsREG} and \Cref{fig:reductions}  summarize the relationship between the studied statements, in Reverse Mathematics and over strong Weihrauch reducibility, respectively.


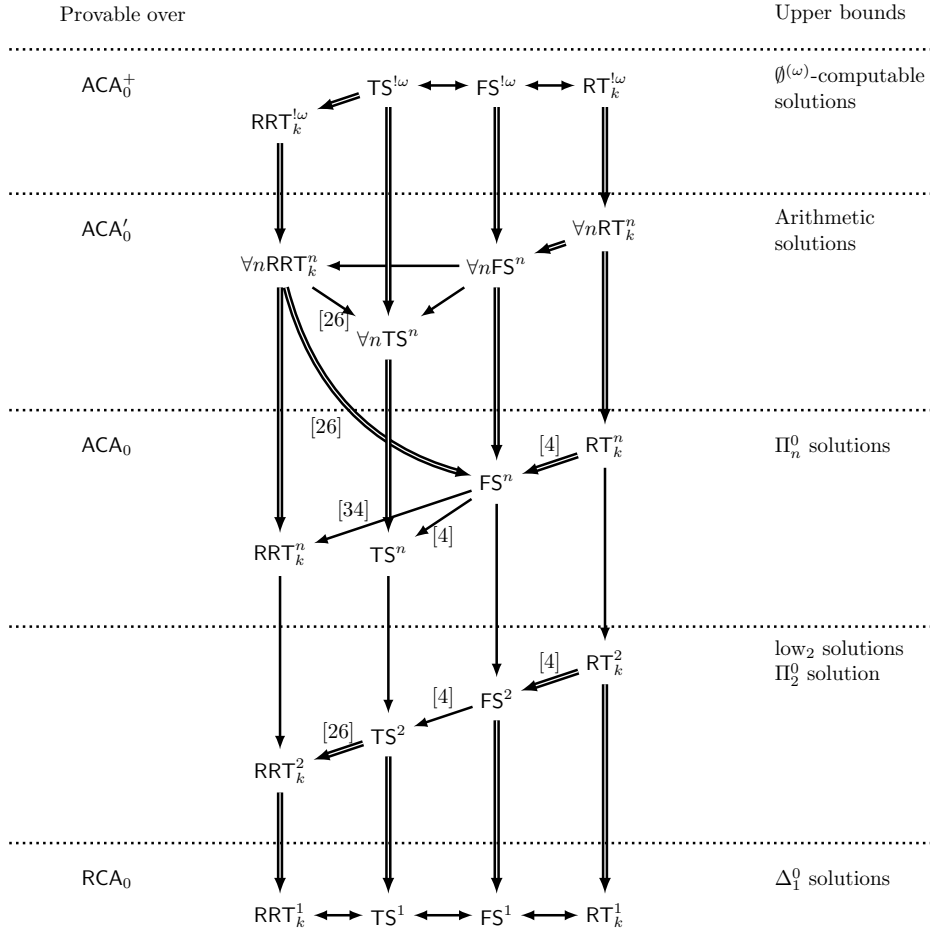
\begin{figure}[htbp]
\begin{center}
\scalebox{.8}{
\begin{tikzpicture}[x=1.8cm, y=1.2cm,
	node/.style={minimum size=2em},
	impl/.style={draw,very thick,-latex},
        strimpl/.style={draw,double,very thick,-latex},
	equiv/.style={draw, very thick, latex-latex},
	caption/.style={text width=3cm}]

        \node[node] (rt1) at (4,-1) {$\RT^1_k$};
        \node[node] (fs1) at (3,-1) {$\FS^1$};
        \node[node] (ts1) at (2,-1) {$\TS^1$};
        \node[node] (rrt1) at (1,-1) {$\RRT^1_k$};

        \node[node] (rt2) at (4,2.5) {$\RT^2_k$};
        \node[node] (fs2) at (3,2) {$\FS^2$};
        \node[node] (ts2) at (2,1.5) {$\TS^2$};
        \node[node] (rrt2) at (1,1) {$\RRT^2_k$};

        \node[node] (rtn) at (4,5.5) {$\RT^n_k$};
        \node[node] (fsn) at (3,5) {$\FS^n$};
        \node[node] (tsn) at (2,4) {$\TS^n$};
        \node[node] (rrtn) at (1,4) {$\RRT^n_k$};

        \node[node] (rtnall) at (4,8.5) {$\forall n\RT^n_k$};
        \node[node] (fsnall) at (3,8) {$\forall n\FS^n$};
        \node[node] (tsnall) at (2,7) {$\forall n\TS^n$};
        \node[node] (rrtnall) at (1,8) {$\forall n\RRT^n_k$};

        \node[node] (rtomega) at (4,10.5) {$\RT^{!\omega}_k$};
        \node[node] (fsomega) at (3,10.5) {$\FS^{!\omega}$};
        \node[node] (tsomega) at (2,10.5) {$\TS^{!\omega}$};
        \node[node] (rrtomega) at (1,10) {$\RRT^{!\omega}_k$};

	\draw[equiv] (rt1) -- (fs1);
        \draw[equiv] (fs1) -- (ts1);
        \draw[equiv] (ts1) -- (rrt1);

        \draw[strimpl] (rt2) -- (fs2) node[midway, above] {\cite{Cholak_Giusto_Hirst_Jockusch_2005}};
        \draw[impl] (fs2) -- (ts2) node[midway, above] {\cite{Cholak_Giusto_Hirst_Jockusch_2005}};
        \draw[strimpl] (ts2) -- (rrt2) node[midway, above] {\cite{patey2015somewhere}};
        
        \draw[strimpl] (rt2) -- (rt1);
        \draw[strimpl] (fs2) -- (fs1);
        \draw[strimpl] (ts2) -- (ts1);
        \draw[strimpl] (rrt2) -- (rrt1);

        \draw[impl] (rtn) -- (rt2);
        \draw[impl] (fsn) -- (fs2);
        \draw[impl] (tsn) -- (ts2);
        \draw[impl] (rrtn) -- (rrt2);

        \draw[strimpl] (rtn) -- (fsn) node[midway, above] {\cite{Cholak_Giusto_Hirst_Jockusch_2005}};
        \draw[impl] (fsn) -- (tsn) node[midway, below] {\cite{Cholak_Giusto_Hirst_Jockusch_2005}};
        \draw[impl] (fsn) -- (rrtn) node[near end, above] {\cite{wang2014some}};

        \draw[strimpl] (rtnall) -- (fsnall);
        \draw[impl] (fsnall) -- (tsnall);
        \draw[impl] (fsnall) -- (rrtnall);
        \draw[impl] (rrtnall) -- (tsnall) node[midway, below] {\cite{patey2015somewhere}};
        \draw[strimpl] (rrtnall) edge[bend right, double] node[midway, below left] {\cite{patey2015somewhere}} (fsn);

        \draw[strimpl] (rtnall) -- (rtn);
        \draw[strimpl] (fsnall) -- (fsn);
        \draw[strimpl] (tsnall) -- (tsn);
        \draw[strimpl] (rrtnall) -- (rrtn);

        \draw[equiv] (rtomega) -- (fsomega);
        \draw[equiv] (fsomega) -- (tsomega);
        \draw[strimpl] (tsomega) -- (rrtomega);

        \draw[strimpl] (rtomega) -- (rtnall);
        \draw[strimpl] (fsomega) -- (fsnall);
        \draw[strimpl] (tsomega) -- (tsnall);
        \draw[strimpl] (rrtomega) -- (rrtnall);

	\draw[very thick, dotted] (-1.5, 0) -- (7,0);
	\node[caption] at (6.4,-0.5) {$\Delta^0_1$ solutions};
        \node[caption] at (0,-0.5) {$\RCA_0$};
	
	\draw[very thick, dotted] (-1.5,3) -- (7,3);
	\node[caption, text width=3cm] at (6.4,2.5) {low${}_2$ solutions\\ $\Pi^0_2$ solution};
	
	\draw[very thick, dotted] (-1.5,6) -- (7,6);
	\node[caption] at (6.4,5.5) {$\Pi^0_n$ solutions};
        \node[caption] at (0,5.5) {$\ACA_0$};

	\draw[very thick, dotted] (-1.5,9) -- (7,9);
	\node[caption] at (6.4,8.5) {Arithmetic\\ solutions};
        \node[caption] at (0,8.5) {$\ACA_0'$};

	\draw[very thick, dotted] (-1.5,11) -- (7,11);
	\node[caption] at (6.4,10.5) {$\emptyset^{(\omega)}$-computable\\ solutions};
        \node[caption] at (0,10.5) {$\ACA_0^+$};

        \node[caption] at (-0.2,11.5) {Provable over};
        \node[caption] at (6.4,11.5) {Upper bounds};
	
\end{tikzpicture}}
\end{center}
\caption{Studied statements from the perspectives of Computability Theory and Reverse Mathematics. A simple (double) arrow represents a (strict) implication over $\RCA_0$.}
\label{fig:implicationsREG}
\end{figure}


\begin{figure}
\centering
\[\begin{tikzcd}[row sep=2em,column sep=2.5em]
\RRT^{!\omega}_k  \arrow[d] & \arrow[l, Rightarrow] \RT^{!\omega}_k \arrow[d, Rightarrow] \arrow[r] &  \RT^{!\omega}_2  \arrow[d, Rightarrow] \arrow[r] &  \FS^{!\omega}  \arrow[dd, Rightarrow] \arrow[r] & \TS^{!\omega} \arrow[ddd, Rightarrow] \\
\RRT^n_k & \arrow[l, Rightarrow] \RT^n_k  \arrow[dr, Rightarrow]   & \RRT^{2n+1}_2 \arrow[dr]  &    &  \\
&   & \RT^n_{2n+2}  \arrow[r] & \FS^n \arrow[d]  &   \\
 & & &\RRT^{n}_2 \arrow[r]  & \TS^{n-1}
\end{tikzcd}\]
\caption{Studied statements from the perspective of Weihrauch analysis. A simple (double) arrow represents a (strict) strong Weihrauch reduction.}\label{fig:reductions}
\end{figure}
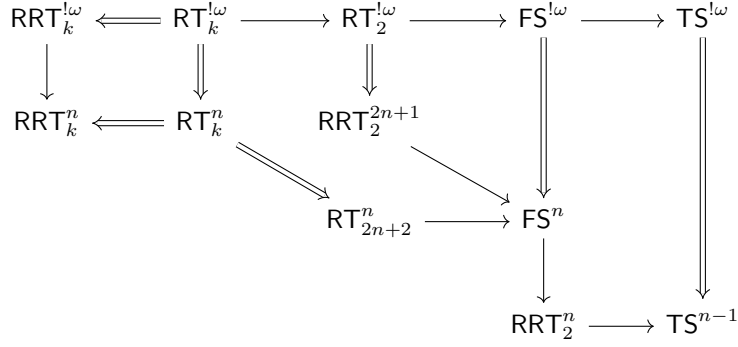

Many questions remain open around the generalization of combinatorial statements to exactly $\omega$-large sets and barriers.


By \Cref{thm:largerrt}, $\RRT^{!\omega}_k \leq_\sW \RT^{!\omega}_k$, so every computable instance of $\RRT^{!\omega}_k$ admits a $\emptyset^{(\omega)}$-computable solution. By \Cref{fs-omega-computes-omega-jump}, here exists a computable instance of $\FS^{!\omega}$ such that every solution computes $\emptyset^{(\omega)}$. It follows that $\RRT^{!\omega}_k$ is computably reducible to $\FS^{!\omega}$ in the sense of Dzhafarov~\cite{dzhafarov2016strong}. We gave a direct combinatorial reduction in \Cref{prop:rrt-omega-2-sw-fs-omega} in the case $k = 2$, and leave the general case open.

\begin{question}
Does $\RRT^{!\omega}_k \leq_\sW \FS^{!\omega}$ for every~$k \in \NN^+$?
\end{question}


While $\RT^{!\omega}, \FS^{!\omega}$ and $\TS^{!\omega}$
code $\emptyset^{(\omega)}$ and are all equivalent to $\ACA_0^+$, we 
don't know of a direct proof of $\RT^{!\omega}_2$ from either $\TS^{!\omega}$ or $\FS^{!\omega}$. More precisely, 
we don't know if $\RT^{!\omega}_2$ is reducible to either $\TS^{!\omega}$ or $\FS^{!\omega}$.

\begin{question}
Does $\RT^{!\omega}_2 \leq_\sW \TS^{!\omega}$? Does $\RT^{!\omega}_2 \leq_\sW \FS^{!\omega}$?
\end{question}

A natural continuation of the line of research of the present paper is to consider the Free Set, Thin Set and Rainbow Ramsey theorems for arbitrary barriers and to inquire into their effective and logical strength. We plan to give in future work a complete layered analysis of the strength of these principles based on the complexity of the barrier. The first and second author have obtained weak anti-basis results relative to the hyperarithmetical hierarchy for the Free Set, Thin Set and Rainbow Ramsey theorems for barriers along the lines of Clote's \cite{clote1984recursion} results for the Barrier Ramsey's Theorem, yet several questions remain to be answered to get a full picture.

In particular, the analysis of the Rainbow Ramsey theorem for barriers revealed a subtlety in the correspondence between the order type of a barrier and the computability-theoretic analysis of the corresponding theorem. Indeed, $\RRT^B_k$ admits strong cone avoidance when $B$ is a computably $\omega$-bounded barrier, while it does not in general when $B$ is a computable barrier of order type $\omega^\omega$ (or equivalently an $\omega$-bounded barrier). This subtle distinction does not arise in the analysis of Ramsey's theorem for barriers. 

Barriers or order type $\omega^\omega$ admit a simple combinatorial characterization as the $\omega$-bounded barriers. Then, computably $\omega$-bounded barriers can be considered as barriers of \emph{effective order type $\omega^\omega$.} Is there an appropriate counterpart to the notion of \qt{effective order type} for larger ordinals?

Clote~\cite{clote1984recursion} proved lower bounds on the Barrier Ramsey Theorem by defining his own notion of \emph{canonical barrier} for every order type. This notion was also used in~\cite{clote1986generalization} to prove a generalization of Shoenfield's limit lemma to the hyperarithmetic hierarchy. On the other hand, Carlucci and Zdanowski~\cite{carlucci2014strength} showed that in the case of barriers of order type $\omega^\omega$, the lower bounds could be witnessed by the Schreier barrier, that is, the barrier of exactly $\omega$-large sets. The notion of $\omega$-large set admits a natural generalization to any computable ordinal $\alpha$, called $\alpha$-largeness
(see H\'ajek and Pudl\'ak~\cite{hajek1998metamathematics}). It is thus natural to ask whether exact $\alpha$-largeness can be used instead of Clote's canonical barriers to witness his lower bounds for the Barrier Ramsey Theorem and to give a layered analysis of the Free Set, Thin Set and Rainbow Ramsey Theorem for barriers. We plan to address these questions in future work. 




\bigskip
\begin{center}
\textbf{Acknowledgements}
\end{center}
The first author was partially supported by the grant {\em Logical Methods in Combinatorics} (28094), PRIN 2022.

The authors are thankful to the anonymous reviewers for their careful reading and improvement suggestions.

\bibliographystyle{plain}
\bibliography{biblio}

@book {hajek1998metamathematics,
    AUTHOR = {H\'{a}jek, Petr and Pudl\'{a}k, Pavel},
     TITLE = {Metamathematics of first-order arithmetic},
    SERIES = {Perspectives in Mathematical Logic},
      NOTE = {Second printing},
 PUBLISHER = {Springer-Verlag, Berlin},
      YEAR = {1998},
     PAGES = {xiv+460},
      ISBN = {3-540-63648-X},
   MRCLASS = {03-02 (03D15 03F30 03H15 11U09 11U10 68Q15)},
  MRNUMBER = {1748522},
}

@book {hirschfeldt2015slicing,
    AUTHOR = {Hirschfeldt, Denis R.},
     TITLE = {Slicing the truth},
    SERIES = {Lecture Notes Series. Institute for Mathematical Sciences.
              National University of Singapore},
    VOLUME = {28},
      NOTE = {On the computable and reverse mathematics of combinatorial
              principles,
              Edited and with a foreword by Chitat Chong, Qi Feng, Theodore
              A. Slaman, W. Hugh Woodin and Yue Yang},
 PUBLISHER = {World Scientific Publishing Co. Pte. Ltd., Hackensack, NJ},
      YEAR = {2015},
     PAGES = {xvi+214},
      ISBN = {978-981-4612-61-6},
   MRCLASS = {03-02 (03B30 03F35)},
  MRNUMBER = {3244278},
MRREVIEWER = {Fran\c{c}ois G. Dorais},
}

@incollection {seetapun1995strength,
    AUTHOR = {Seetapun, David and Slaman, Theodore A.},
     TITLE = {On the strength of {R}amsey's theorem},
      NOTE = {Special Issue: Models of arithmetic},
   JOURNAL = {Notre Dame J. Formal Logic},
  FJOURNAL = {Notre Dame Journal of Formal Logic},
    VOLUME = {36},
      YEAR = {1995},
    NUMBER = {4},
     PAGES = {570--582},
      ISSN = {0029-4527},
   MRCLASS = {03F35 (03C62)},
  MRNUMBER = {1368468},
MRREVIEWER = {Roman Murawski},
       DOI = {10.1305/ndjfl/1040136917},
       URL = {https://doi.org/10.1305/ndjfl/1040136917},
}

@article{liu2012rt22,
  title={{RT}22 does not imply {WKL}0},
  author={Liu, Jiayi},
  journal={The Journal of Symbolic Logic},
  volume={77},
  number={2},
  pages={609--620},
  year={2012},
  publisher={Cambridge University Press}
}

@article{patey2015somewhere,
  title={Somewhere over the rainbow Ramsey theorem for pairs},
  author={Patey, Ludovic},
  journal={arXiv preprint arXiv:1501.07424},
  year={2015}
}

@article {cholak2020thin,
    AUTHOR = {Cholak, Peter and Patey, Ludovic},
     TITLE = {Thin set theorems and cone avoidance},
   JOURNAL = {Trans. Amer. Math. Soc.},
  FJOURNAL = {Transactions of the American Mathematical Society},
    VOLUME = {373},
      YEAR = {2020},
    NUMBER = {4},
     PAGES = {2743--2773},
      ISSN = {0002-9947},
   MRCLASS = {03B30 (05A10)},
  MRNUMBER = {4069232},
MRREVIEWER = {Huishan Wu},
       DOI = {10.1090/tran/7987},
       URL = {https://doi.org/10.1090/tran/7987},
}

@article {wang2014some,
    AUTHOR = {Wang, Wei},
     TITLE = {Some logically weak {R}amseyan theorems},
   JOURNAL = {Adv. Math.},
  FJOURNAL = {Advances in Mathematics},
    VOLUME = {261},
      YEAR = {2014},
     PAGES = {1--25},
      ISSN = {0001-8708},
   MRCLASS = {03F35 (03B30)},
  MRNUMBER = {3213294},
MRREVIEWER = {Alberto Marcone},
       DOI = {10.1016/j.aim.2014.05.003},
       URL = {https://doi.org/10.1016/j.aim.2014.05.003},
}

@article {dorais2016uniform,
    AUTHOR = {Dorais, Fran\c{c}ois G. and Dzhafarov, Damir D. and Hirst, Jeffry
              L. and Mileti, Joseph R. and Shafer, Paul},
     TITLE = {On uniform relationships between combinatorial problems},
   JOURNAL = {Trans. Amer. Math. Soc.},
  FJOURNAL = {Transactions of the American Mathematical Society},
    VOLUME = {368},
      YEAR = {2016},
    NUMBER = {2},
     PAGES = {1321--1359},
      ISSN = {0002-9947},
   MRCLASS = {03B30 (03D30 03D32 03D80 03F35 05D10 05D15 05D40)},
  MRNUMBER = {3430365},
MRREVIEWER = {Denis R. Hirschfeldt},
       DOI = {10.1090/tran/6465},
       URL = {https://doi.org/10.1090/tran/6465},
}

@article{car-main-zda2024,
  title={Reductions of well-ordering principles to combinatorial theorems},
  author={Carlucci, Lorenzo and Mainardi, Leonardo and Zdanowski, Konread},
  journal={arXiv preprint arXiv:2401.04451},
  year={2024}
}

@article {carlucci2014strength,
    AUTHOR = {Carlucci, Lorenzo and Zdanowski, Konrad},
     TITLE = {The strength of {R}amsey's theorem for coloring relatively
              large sets},
   JOURNAL = {J. Symb. Log.},
  FJOURNAL = {The Journal of Symbolic Logic},
    VOLUME = {79},
      YEAR = {2014},
    NUMBER = {1},
     PAGES = {89--102},
      ISSN = {0022-4812},
   MRCLASS = {03F35 (03B30 05D10)},
  MRNUMBER = {3226013},
MRREVIEWER = {Alberto Marcone},
       DOI = {10.1017/jsl.2013.27},
       URL = {https://doi.org/10.1017/jsl.2013.27},
}

@book{Todorcevic+2010,
url = {https://doi.org/10.1515/9781400835409},
title = {Introduction to Ramsey Spaces},
author = {Stevo Todorcevic},
publisher = {Princeton University Press},
address = {Princeton},
doi = {doi:10.1515/9781400835409},
isbn = {9781400835409},
year = {2010},
lastchecked = {2024-11-14}
}

@incollection{Paris1977-PARAMI,
	author = {Jeff Paris and Leo Harrington},
	booktitle = {Handbook of mathematical logic},
	editor = {Jon Barwise},
	pages = {90--1133},
	publisher = {North-Holland},
	title = {A Mathematical Incompleteness in {Peano} {Arithmetic}},
	year = {1977}
}

@incollection {Cholak_Giusto_Hirst_Jockusch_2005,
    AUTHOR = {Cholak, Peter A. and Giusto, Mariagnese and Hirst, Jeffry L.
              and Jockusch, Jr., Carl G.},
     TITLE = {Free sets and reverse mathematics},
 BOOKTITLE = {Reverse mathematics 2001},
    SERIES = {Lect. Notes Log.},
    VOLUME = {21},
     PAGES = {104--119},
 PUBLISHER = {Assoc. Symbol. Logic, La Jolla, CA},
      YEAR = {2005},
   MRCLASS = {03F35 (03B30 03D28 03D80)},
  MRNUMBER = {2185429},
}

@article{JockuschRamsey_1972,
 ISSN = {00224812},
 URL = {http://www.jstor.org/stable/2272972},
 author = {Carl G. Jockusch},
 journal = {The Journal of Symbolic Logic},
 number = {2},
 pages = {268--280},
 publisher = {Association for Symbolic Logic},
 title = {Ramsey's Theorem and Recursion Theory},
 urldate = {2024-11-14},
 volume = {37},
 year = {1972}
}

@article{Liu2022-LIUTRM,
	author = {Lu Liu and Ludovic Patey},
	doi = {10.1017/jsl.2021.98},
	journal = {Journal of Symbolic Logic},
	number = {1},
	pages = {313--346},
	title = {{The Reverse Mathematics of the Thin Set and {E}rd\H{o}s-{M}oser Theorems}},
	volume = {87},
	year = {2022}
}

@article{CsimaMileti_2009,
 ISSN = {00224812, 19435886},
 URL = {http://www.jstor.org/stable/40378269},
 author = {Barbara F. Csima and Joseph R. Mileti},
 journal = {The Journal of Symbolic Logic},
 number = {4},
 pages = {1310--1324},
 publisher = {[Association for Symbolic Logic, Cambridge University Press]},
 title = {The Strength of the Rainbow Ramsey Theorem},
 urldate = {2024-11-14},
 volume = {74},
 year = {2009}
}

@article {assous1974caracterisation,
    AUTHOR = {Assous, Marc},
     TITLE = {Caract\'{e}risation du type d'ordre des barri\`eres de
              {N}ash-{W}illiams},
   JOURNAL = {Publ. D\'{e}p. Math. (Lyon)},
  FJOURNAL = {Publications du D\'{e}partement de Math\'{e}matiques. Facult\'{e} des
              Sciences de Lyon},
    VOLUME = {11},
      YEAR = {1974},
    NUMBER = {4},
     PAGES = {89--106},
      ISSN = {0076-1656},
   MRCLASS = {06A05},
  MRNUMBER = {366758},
MRREVIEWER = {C. St. J. A. Nash-Williams},
}

@incollection {marcone2005wqo,
    AUTHOR = {Marcone, Alberto},
     TITLE = {Wqo and bqo theory in subsystems of second order arithmetic},
 BOOKTITLE = {Reverse mathematics 2001},
    SERIES = {Lect. Notes Log.},
    VOLUME = {21},
     PAGES = {303--330},
 PUBLISHER = {Assoc. Symbol. Logic, La Jolla, CA},
      YEAR = {2005},
   MRCLASS = {03F35},
  MRNUMBER = {2185443},
}

@article {liu2015cone,
    AUTHOR = {Liu, Lu},
     TITLE = {Cone avoiding closed sets},
   JOURNAL = {Trans. Amer. Math. Soc.},
  FJOURNAL = {Transactions of the American Mathematical Society},
    VOLUME = {367},
      YEAR = {2015},
    NUMBER = {3},
     PAGES = {1609--1630},
      ISSN = {0002-9947},
   MRCLASS = {03B30 (03C62 03D32 03F35 28A78 68Q30 68Q80)},
  MRNUMBER = {3286494},
MRREVIEWER = {Damir D. Dzhafarov},
       DOI = {10.1090/S0002-9947-2014-06049-2},
       URL = {https://doi.org/10.1090/S0002-9947-2014-06049-2},
}

@article {patey2017iterative,
    AUTHOR = {Patey, Ludovic},
     TITLE = {Iterative forcing and hyperimmunity in reverse mathematics},
   JOURNAL = {Computability},
  FJOURNAL = {Computability. The Journal of the Association CiE},
    VOLUME = {6},
      YEAR = {2017},
    NUMBER = {3},
     PAGES = {209--221},
      ISSN = {2211-3568},
   MRCLASS = {03B30 (03C62 03D30 03D80 03E40 03F35)},
  MRNUMBER = {3689068},
MRREVIEWER = {Robert S. Lubarsky},
       DOI = {10.3233/COM-160062},
       URL = {https://doi.org/10.3233/COM-160062},
}

@article {nash-williams1968better,
    AUTHOR = {Nash-Williams, C. St. J. A.},
     TITLE = {On better-quasi-ordering transfinite sequences},
   JOURNAL = {Proc. Cambridge Philos. Soc.},
  FJOURNAL = {Proceedings of the Cambridge Philosophical Society},
    VOLUME = {64},
      YEAR = {1968},
     PAGES = {273--290},
      ISSN = {0008-1981},
   MRCLASS = {04.60 (05.00)},
  MRNUMBER = {221949},
MRREVIEWER = {W. T. Tutte},
       DOI = {10.1017/s030500410004281x},
       URL = {https://doi.org/10.1017/s030500410004281x},
}

@article{pouzet1972premeilleurordres,
  title={Sur les pr{\'e}meilleurordres},
  author={Pouzet, Maurice},
  journal={Annales de l'institut Fourier},
  volume={22},
  number={2},
  pages={1--19},
  year={1972}
}

@article{clote1984recursion,
  title = {A Recursion Theoretic Analysis of the Clopen {{Ramsey}} Theorem},
  author = {Clote, Peter},
  year = {1984},
  journal = {The Journal of symbolic logic},
  volume = {49},
  number = {2},
  pages = {376--400},
  publisher = {Cambridge University Press}
}

@article{clote1986generalization,
  title = {A Generalization of the Limit Lemma and Clopen Games},
  author = {Clote, Peter},
  year = {1986},
  journal = {Journal of Symbolic Logic},
  volume = {51},
  number = {2},
  pages = {273--291},
  issn = {0022-4812},
  doi = {10.2307/2274051},
  fjournal = {The Journal of Symbolic Logic},
  mrclass = {03D55 (03D25)},
  mrnumber = {840405}
}

@article{galvinBorelSetsRamsey1973,
  title = {Borel Sets and {{Ramsey}}'s Theorem},
  author = {Galvin, Fred and Prikry, Karel},
  year = {1973},
  journal = {The Journal of Symbolic Logic},
  volume = {38},
  pages = {193--198},
  issn = {0022-4812},
  mrnumber = {0337630}
}

@book {Dza-Mum:22,
    AUTHOR = {Dzhafarov, Damir D. and Mummert, Carl},
     TITLE = {Reverse mathematics---problems, reductions, and proofs},
    SERIES = {Theory and Applications of Computability},
 PUBLISHER = {Springer, Cham},
      YEAR = {[2022] \copyright 2022},
     PAGES = {xix+488},
      ISBN = {978-3-031-11366-6; 978-3-031-11367-3},
   MRCLASS = {03-02 (03B30 03F35)},
  MRNUMBER = {4472209},
MRREVIEWER = {Huishan Wu},
       DOI = {10.1007/978-3-031-11367-3},
       URL = {https://doi.org/10.1007/978-3-031-11367-3},
}

@article {Pat:16,
    AUTHOR = {Patey, Ludovic},
     TITLE = {The weakness of being cohesive, thin or free in reverse
              mathematics},
   JOURNAL = {Israel J. Math.},
  FJOURNAL = {Israel Journal of Mathematics},
    VOLUME = {216},
      YEAR = {2016},
    NUMBER = {2},
     PAGES = {905--955},
      ISSN = {0021-2172},
   MRCLASS = {03B30 (03F35)},
  MRNUMBER = {3557471},
MRREVIEWER = {Fran\c{c}ois G. Dorais},
       DOI = {10.1007/s11856-016-1433-3},
       URL = {https://doi.org/10.1007/s11856-016-1433-3},
}

@book {SIM:SOSOA,
    AUTHOR = {Simpson, Stephen G.},
     TITLE = {Subsystems of second order arithmetic},
    SERIES = {Perspectives in Logic},
   EDITION = {Second},
 PUBLISHER = {Cambridge University Press, Cambridge; Association for
              Symbolic Logic, Poughkeepsie, NY},
      YEAR = {2009},
     PAGES = {xvi+444},
      ISBN = {978-0-521-88439-6},
   MRCLASS = {03F35 (03-02 03B30)},
  MRNUMBER = {2517689},
       DOI = {10.1017/CBO9780511581007},
       URL = {https://doi.org/10.1017/CBO9780511581007},
}

@article {dzhafarov2016strong,
    AUTHOR = {Dzhafarov, Damir D.},
     TITLE = {Strong reductions between combinatorial principles},
   JOURNAL = {J. Symb. Log.},
  FJOURNAL = {The Journal of Symbolic Logic},
    VOLUME = {81},
      YEAR = {2016},
    NUMBER = {4},
     PAGES = {1405--1431},
      ISSN = {0022-4812},
   MRCLASS = {03B30 (03F35 05D05)},
  MRNUMBER = {3579116},
MRREVIEWER = {Alberto Marcone},
       DOI = {10.1017/jsl.2016.1},
       URL = {https://doi.org/10.1017/jsl.2016.1},
}

@article {Pud-Rod:82,
    AUTHOR = {Pudl\'{a}k, Pavel and R\"{o}dl, Vojt\v{e}ch},
     TITLE = {Partition theorems for systems of finite subsets of integers},
   JOURNAL = {Discrete Math.},
  FJOURNAL = {Discrete Mathematics},
    VOLUME = {39},
      YEAR = {1982},
    NUMBER = {1},
     PAGES = {67--73},
      ISSN = {0012-365X},
   MRCLASS = {05A17 (04A20)},
  MRNUMBER = {677888},
MRREVIEWER = {E. C. Milner},
       DOI = {10.1016/0012-365X(82)90041-3},
       URL = {https://doi.org/10.1016/0012-365X(82)90041-3},
}

@article {Far-Neg:08,
    AUTHOR = {Farmaki, V. and Negrepontis, S.},
     TITLE = {Schreier sets in {R}amsey theory},
   JOURNAL = {Trans. Amer. Math. Soc.},
  FJOURNAL = {Transactions of the American Mathematical Society},
    VOLUME = {360},
      YEAR = {2008},
    NUMBER = {2},
     PAGES = {849--880},
      ISSN = {0002-9947},
   MRCLASS = {05D10},
  MRNUMBER = {2346474},
MRREVIEWER = {N. Hindman},
       DOI = {10.1090/S0002-9947-07-04323-1},
       URL = {https://doi.org/10.1090/S0002-9947-07-04323-1},
}

@incollection {Fri-Sim:00,
    AUTHOR = {Friedman, Harvey and Simpson, Stephen G.},
     TITLE = {Issues and problems in reverse mathematics},
 BOOKTITLE = {Computability theory and its applications ({B}oulder, {CO},
              1999)},
    SERIES = {Contemp. Math.},
    VOLUME = {257},
     PAGES = {127--144},
 PUBLISHER = {Amer. Math. Soc., Providence, RI},
      YEAR = {2000},
   MRCLASS = {03F35 (03B30 05D10 26E40 46B99)},
  MRNUMBER = {1770738},
MRREVIEWER = {Kazuyuki Tanaka},
       DOI = {10.1090/conm/257/04031},
       URL = {https://doi.org/10.1090/conm/257/04031},
}

@article {Kur:51,
    AUTHOR = {Kuratowski, Casimir},
     TITLE = {Sur une caract\'{e}risation des alephs},
   JOURNAL = {Fund. Math.},
  FJOURNAL = {Polska Akademia Nauk. Fundamenta Mathematicae},
    VOLUME = {38},
      YEAR = {1951},
     PAGES = {14--17},
      ISSN = {0016-2736},
   MRCLASS = {27.2X},
  MRNUMBER = {48518},
MRREVIEWER = {F. Bagemihl},
       DOI = {10.4064/fm-38-1-14-17},
       URL = {https://doi.org/10.4064/fm-38-1-14-17},
}

@article{montalban_open_2011,
    title = {Open questions in reverse mathematics},
    volume = {17},
    number = {03},
    journal = {Bulletin of Symbolic Logic},
    author = {Montalbán, Antonio},
    year = {2011},
    note = {Publisher: Cambridge Univ Press},
    pages = {431--454},
}

@book {gentzen1967wiederspruchsfreiheit,
    AUTHOR = {Gentzen, Gerhard},
     TITLE = {Die {W}iderspruchsfreiheit der reinen {Z}ahlentheorie},
    SERIES = {Libelli, Band 185},
 PUBLISHER = {Wissenschaftliche Buchgesellschaft, Darmstadt},
      YEAR = {1967},
     PAGES = {ii+73},
   MRCLASS = {02.50 (10.00)},
  MRNUMBER = {221921},
}

@article{carlucci2025ramsey, 
    title={Ramsey-like theorems for the Schreier barrier}, 
    DOI={10.1017/jsl.2025.10104}, 
    journal={The Journal of Symbolic Logic}, 
    author={Carlucci, Lorenzo and Gjetaj, Oriola and Le Houérou, Quentin and Levy Patey, Ludovic}, 
    year={2025}, 
    pages={1–29}
}

\end{document}